\documentclass[reqno,11pt,twoside,english]{article}
\usepackage{ amsmath, amsfonts, amssymb, amsthm, amscd}
\usepackage[T1]{fontenc}
\usepackage[english]{babel}
\usepackage[noinfoline]{imsart}

\DeclareMathSymbol{\leqslant}{\mathalpha}{AMSa}{"36} 
\DeclareMathSymbol{\geqslant}{\mathalpha}{AMSa}{"3E} 
\renewcommand{\leq}{\;\leqslant\;}                   
\renewcommand{\geq}{\;\geqslant\;}                   

\newtheorem{Th}{Theorem}
\newtheorem{Le}[Th]{Lemma}
\newtheorem{Pro}[Th]{Proposition}
\newtheorem{Def}[Th]{Definition}
\newtheorem{Cor}[Th]{Corollary}
\newtheorem{Rq}{Remark}

\newcommand{\cF}{\ensuremath{\mathcal F}}

\newcommand{\cH}{\ensuremath{\mathcal H}}

\newcommand{\cO}{\ensuremath{\mathcal O}}

\newcommand{\bbR}{{\ensuremath{\mathbb R}} }

\newcommand{\ub}{\bar{u}}
\newcommand{\fb}{\bar{f}}
\newcommand{\gb}{\bar{g}}
\newcommand{\hb}{\bar{h}}
\newcommand{\N}{\mathbb{N}}

\newcommand{\be}{\begin{equation}}
\newcommand{\ee}{\end{equation}}
\newcommand{\beq}{\begin{eqnarray}}
\newcommand{\eeq}{\end{eqnarray}}

\newcommand{\1}{{1} \hspace{-0.25 em}{\rm I}}

\newcommand{\eps}{\varepsilon}

\newcommand{\Ne}{\N^{\ast}}

\newcommand{\HH}{{\cal H}}
\newcommand{\HHl}{{\cal H}_{loc}}

\newcommand{\OO}{{\cal O}}
\newcommand{\PP}{{\cal P}}

\newcommand{\R}{\mathbb{R}}

\newcommand{\dis}{\displaystyle}

\newcommand{\ced}{\end{proof}}

\newcommand{\p}{^{\prime}}
\newcommand{\vphi}{\varphi}

\newcommand{\disp}{\displaystyle}

\newcommand{\tu}{\tilde{u}}

\newcommand{\txi}{\tilde{\xi}}
\newcommand{\tw}{\tilde{w}}
\newcommand{\vep}{\varepsilon}
\newcommand{\tA}{\tilde{A}}
\newcommand{\tz}{\tilde{z}}

\begin{document}
\begin{frontmatter}
\title{Maximum principle for  quasilinear SPDE's on a bounded  domain without regularity assumptions  }
\date{}
\runtitle{}
\author{\fnms{Laurent}
 \snm{DENIS}\corref{}\ead[label=e1]{ldenis@univ-evry.fr}}
\thankstext{T1}{The work of the first author is supported by the chair \textit{risque de cr\'edit}, F\'ed\'eration bancaire Fran\c{c}aise}
\address{D\'epartement de Math\'ematiques\\ Equipe Analyse et Probabilit\'es\\Universit\'e
d'Evry-Val-d'Essonne\\23 Boulevard de France\\91037 EVRY Cedex
-FRANCE\\\printead{e1}}

\author{\fnms{Anis}
 \snm{MATOUSSI}\corref{}\ead[label=e2]{anis.matoussi@univ-lemans.fr}}
\thankstext{t2}{The work of the second author is  supported  by the Chair \textit{Financial Risks of the Risk Foundation} sponsored by Société Générale,  CMAP \'Ecole Polytechnique. }
\address{Laboratoire Manceau de Mathématiques \\
Université du Maine\\Avenue Olivier Messiaen\\72 085 LE MANS Cedex - FRANCE \\\printead{e2}}
\runauthor{L. Denis and A. Matoussi}

\begin{abstract}

We prove a maximum principle  for  local
solutions of  quasi-linear parabolic stochastic PDEs, with non-homogeneous second order operator on a bounded domain and driven by a space-time white noise. Our method based on an approximation  of the domain and the coefficients of the operator, does not require  regularity assumptions.  As in previous works \cite{DMS1, DMS2} the results  are
consequences  of It\^{o}'s formula and estimates for the positive
part of  local solutions which are non-positive on the lateral
boundary.

\end{abstract}


\end{frontmatter}
\section{Introduction}

In the theory of deterministic Partial Differential Equations, the maximum
principle plays an important role since it gives a relation between the bound of
the solution on the boundary and a bound on the whole domain. In the deterministic case, the
maximum principle for quasi-linear parabolic equations was  proved
by Aronson -Serrin (see Theorem 1 of \cite{Aronson}).\\
In a previous work \cite{DMS2}, we have  adapted the method of these authors to the stochastic framework and proved  maximum principle
for SPDE's with homogeneous second order operator and driven by a finite dimensional Brownian motion.
 The aim of the present paper is to generalize  these results to the case of SPDE's with non-homogeneous second order operator and driven by a noise which is  white in time and colored in space. In \cite{DMS1} and \cite{DMS2},   many proofs are based on the notion of semigroup associated to the second order operator and on the regularizing property of the semigroup. But now,  since in this present paper the operator is non homogeneous we can not follow exactly the same proofs and so we work with the Green function associated to the operator and use heavily the results of Aronson \cite{Aronson3} on the existence and the Gaussian estimates of the weak fundamental solution of a parabolic PDE.

More precisely, we study the following stochastic partial
differential equation (in short SPDE) for a real-valued random
field $u_t\left( x\right)
=u\left( t,x\right) ,$%
\begin{equation}
\begin{split}
\label{e0}
  du_t (x)  =
\Big( \sum_{i=1}^d \partial_i &\big[\sum_{j=1}^d\ a_{i,j}(t,x)\partial_j u_t (x)
+g_i(t,x,u_t(x),\nabla u_t (x)) \big]+f(t,x,u_t (x),\nabla u_t (x))
\Big) dt\\
 &  + \sum_{i=1}^{+\infty}h_i(t,x,u_t(x),\nabla u_t(x))\, dB^i_t,\\
\end{split}
\end{equation}  with a given initial condition $u_0=\xi ,$ where $a$ is a time-dependant
symmetric, uniformly elliptic, measurable matrix defined on  some
bounded open domain ${\cal O}\subset \mathbb{R}^d$ and $%
f,g_i,i=1,\cdots ,d,h_j,j=1,2, \cdots$ are nonlinear random functions.\\

This class of SPDE's has been widely studied by many authors (see \cite{KR}, \cite{CM}, \cite{SW},....) but in all these references,  regularity assumptions are made on the boundary of the domain or on $a$ which permit to use Sobolev embedding theorems or/and  regularity of the Green function.
Since in this work coefficients $a_{i,j}$ and the domain $\cO$ are not smooth, the associated Green function is not regular enough, so one more time we follow the ideas of Aronson (see \cite{Aronson3}). The method consists in  approximating the domain by an increasing sequence of smooth domains and the matrix $a$ by a sequence of smooth matrices. We first prove existence and uniqueness for the SPDE \eqref{e0} with null Dirichlet condition on the boundary. Then we get some estimates of the positive part of a local solution which is non-negative on the boundary and this permits to get a comparison Theorem and a maximum principle, in this part only we assume that the boundary of the domain is Lipschitz . This yields for example the following result:

\begin{Th}
Let $(M_t)_{t\geq 0}$ be an Itô process satisfying some
integrability conditions, $p\geq 2$ and $u$ be a local weak
solution of (\ref{e0}). Assume that $\partial O$ is Lipschitz and
that $u\leq M$ on the parabolic boundary $\left\{ [0,T[\times \partial {\cal %
O}\right\} \cup \left\{ \left\{ 0\right\} \times {\cal O}\right\}
$, then for all $t\in [0,T]$: $$ E\left\| \left( u-M\right)
^{+}\right\| _{\infty ,\infty ;t}^p\le k\left( p,t\right) E\left( \|\xi - M_0\|_{\infty}^p+
\left\| (f^{0,M})^{+}\right\| _{\theta ,t}^{*p}+\left\|
|g^{0,M}|^2\right\| _{\theta ;t}^{*p/2}+\left\|
|h^{0,M}|^2\right\| _{\theta ;t}^{*p/2}\right) $$
where $f^{0,M}(t,x)=f(t,x,M,0),$ $g^{0,M}(t,x)=g(t,x,M,0),$ $%
h^{0,M}(t,x)=h(t,x,M,0)\,$and $k$ is a function which only depends
on the structure constants of the SPDE, $\Vert \cdot \Vert _{\infty
,\infty ;t}$ is the uniform norm on $[0,t]\times {\cal O}$ and
$\left\| \cdot \right\| _{\theta ;t}^{*}$  is a certain norm which
is precisely defined below.
\end{Th}

 For the references concerning the study of the $L^p$ norms w.r.t. the randomness of  uniform norm on the trajectories of a stochastic PDE,  see \cite{DMS2}.
 Let us also mention that some $L^p$-estimates have been established by Kim \cite{Kim} for linear parabolic spde's on Lipschitz domain and that Krylov \cite{Krylov2}  obtained a maximum principle for the same class of  SPDE's.\\
 The paper is organized as follows : in section \ref{preliminaries}
we introduce notations and hypotheses and we take care to detail
the integrability conditions which are used  all along the paper.
In section \ref{DirichletCondition} we establish an It\^{o} formula for the solution and prove existence and uniqueness of this solution with null Dirichlet condition on the boundary. In section  \ref{LocalSolutions}, we prove an It\^{o}'s formula and estimates for the positive part of a local solution which is non-positive on the boundary of the domain and obtain  a
comparison Theorem which leads to our main result: the maximum principle Theorem. The last section is an Appendix devoted to the definitions of some functional spaces that we use and to the proofs of some technical results.

\section{Preliminaries}\label{preliminaries}
\subsection{$L^{p,q}$-spaces}
\label{lpq}
Let ${\cal O}$ be an open bounded domain in $\mathbb{{ R}}^d.$ The space $%
L^2\left( {\cal O}\right) $ is the basic Hilbert space of our
framework and
we employ the usual notation for its scalar product and its norm,%
$$ \left( u,v\right) =\int_{{\cal O}}u\left( x\right) v\left(
x\right) dx,\;\left\| u\right\|=\left( \int_{{\cal O}}u^2\left(
x\right) dx\right) ^{\frac 12}. $$ In general, we shall extend the
notation
\[ (u,v)=\int_{\cO} u(x)v(x)\, dx,\]
where $u$, $v$ are measurable functions defined on $\cO$ such that $uv \in L^1 (\cO )$.\\
The first order Sobolev space
of functions vanishing at the
boundary will be denoted as usual by $H_0^1\left( {\cal O}%
\right) .$ Its natural scalar product and norm are%
$$
\left( u,v\right) _{H_0^1\left( {\cal O}\right) }=\left( u,v\right) +\sum_{i=1}^d\int_{%
{\cal O}} \partial _i u\left( x\right)
\partial _iv\left( x\right) dx,\,\left\| u\right\| _{H_0^1\left(
{\cal O}\right) }=\left( \left\| u\right\| _2^2+\left\| \nabla
u\right\| _2^2\right) ^{\frac 12}. $$ We shall denote  by  $
H_{loc}^1 (\cO)$ the space of functions which are locally square
integrable in $ \mathcal{O}$ and which admit first order derivatives
that are also locally square integrable.\\
Another Hilbert space that we use is the second order Sobolev space $H^2_0 (\cO )$ of  functions vanishing at the
boundary and twice differentiable in the weak sense.

For each $t>0$ and for all real numbers $p,\,q\geq 1$, we denote by $%
L^{p,q}([0,t]\times {\cal O})$ the space of (classes of) measurable
functions $u:[0,t]\times {\cal O}\longrightarrow \mathbb{{R}}$ such
that
$$
\Vert u\Vert _{p,q;\,t}:=\left( \int_0^t\left( \int_{{\cal O}%
}|u(s,x)|^p\,dx\right) ^{q/p}\,ds\right) ^{1/q} $$ is finite. The
limiting cases with $p$ or $q$ taking the value $\infty $ are also
considered with the use of the essential sup norm. 

The space of  measurable functions $u:\mathbb{{R}}_{+}\rightarrow
L^2\left(
{\cal O}\right) $ such that $\left\| u\right\| _{2,2;t}<\infty ,$ for each $%
t\ge 0,$ is denoted by $L_{loc}^2\left( \mathbb{{R}}_{+};L^2\left( {\cal O}%
\right) \right) ,$ where $\R_+$ denotes the set of non-negative real numbers. Similarly, the space $L_{loc}^2\left( \mathbb{{R}}%
_{+};H_0^1\left( {\cal O}\right) \right) $ consists of all
measurable
functions $u: \mathbb{{R}}_{+}\rightarrow H_0^1\left( {\cal O}\right) $ such that%
$$ \left\| u\right\| _{2,2;t}+\left\| \nabla u\right\|
_{2,2;t}<\infty , $$ for any $t\ge 0.$\\

 We recall that the Sobolev inequality states that%
$$ \left\| u\right\| _{2^{*}}\le c_S\left\| \nabla u\right\| _2,
$$ for each $u\in H_0^1\left( {\cal O}\right) ,$ where $c_S>0$ is
a constant
that depends on the dimension and $2^{*}=\frac{2d}{d-2}$ if $d>2,$ while $%
2^{*}$ may be any number in $]2,\infty [$ if $d=2$ and $2^{*}=\infty $ if $%
d=1.$\\ Finally, we introduce the following norm which is obtained by interpolation in $L^{p,q}$-spaces:
$$ \left\| u\right\| _{\#;t}=\left\|
u\right\| _{2,\infty ;t}\vee \left\| u\right\|
_{2^{*},2;t}, $$
and we denote by $L_{\# ;t}$ the set of functions $u$ such that $\left\| u \right\| _{\#;t}$ is finite. Its dual space is a functional space: $L^*_{\#;t}$ equipped with the norm $\parallel\ \parallel^*_{\#;t}$ and we have
\begin{equation}\label{Holder2} \int_0^t\int_{{\cal O}}u\left( s,x\right) v\left( s,x\right)
dxds\le \left\| u\right\| _{\#;t}\left\| v\right\|_{\#;t}^*, \end{equation} for
any $u\in L_{\#;t}$ and $v\in L_{\#;t}^*. $\\
 See Appendix \ref{lpq} for more details on these spaces.
\subsection{Hypotheses and definitions}\label{hypo}

We consider a sequence $((B^i
(t))_{t\geq 0})_{i\in\Ne}$ of
 independent Brownian motions defined on a standard filtered probability space $(\Omega ,\cF ,(\cF )_{t\geq 0}, P)$ satisfying the usual conditions.
 Let  $ a$ be  a measurable and symmetric matrix defined
on $ \R_+\times\cO$ . We assume that there exist positive constants
$ \lambda $ , $ \Lambda $ and $M$ such that for all $\xi \in \bbR^{
d} $ and almost all $(t,x)\in \R_+ \times \cO$:
\begin{equation}
\label{ellip} \lambda |\xi|^2 \, \leq \sum_{i,j} a_{i,j} (t,x)
\xi^i \, \xi^j \, \leq \, \Lambda |\xi|^2  \makebox{
and } |a_{i,j} (t,x)|\leq M.
\end{equation}
 Let $\Delta =\{(t,x,s,y)\in\R_+ \times \cO\times \R_+ \times \cO; t> s\}.$ We denote  by $G:\Delta \mapsto \R_+$ the
  weak fundamental solution of the problem\\
  \begin{equation}
  \label{weak:fundamental}
\partial_t G (t,x;s,y)  -
\sum_{i,j=1}^d \partial_i a_{i,j}(t,x)\partial_j G (t,x;s,y) =0
\end{equation}
with  Dirichlet boundary condition $ G(t,x,s,y) = 0 , \quad
\mbox{for all} \; (t,x) \in \; (s,\, +\infty ) \times \partial \cO
\,$  and where for $i\in \{1,\cdots ,d\}$, $\partial_i$ denotes the 
partial derivative of oder $1$ with respect to $x_i$.\\
Sometimes, for convenience, we shall restrict ourselves to a finite time-interval, that's why we fix a time $T>0$.\\
Following Aronson (\cite{Aronson3}), Theorem 9 (iii) p. 671,  we
have the following estimate: \begin{equation}{\label{Gaussestimate}}
G(t,x,s,y)\leq C (t-s)^{-\frac d2}\exp\{-\varrho
\frac{|x-y|^2}{8(t-s)}\},\end{equation} for all
$(t,x,s,y)\in\Delta$ with $t\leq T$,  where $C$ and $\varrho$ are
positive constants depending only on $T$ and the structure
constants i.e. $\lambda$, $\Lambda$ and the Lebesgue measure of $\cO$. Let us point out that we do not assume that
coefficients are smooth or that $\partial \cO$ is regular.

 We consider predictable random
functions
\begin{eqnarray*}
f &:&\bbR_{+}\times \Omega \times  \cO \times \bbR \times \bbR^{d}
\rightarrow \bbR \;, \\
g &=& (g_1,...,g_d) \, : \, \bbR_{+}\times \Omega \times  \cO
\times \bbR \times \bbR^{d}
\rightarrow \bbR^{d}\;\\
h=(h_1 ,\cdots, h_i ,\cdots ) &:&\bbR_{+}\times \Omega \times  \cO \times \bbR \times \bbR^{d}
\rightarrow \bbR^{\Ne} \;,\\
\end{eqnarray*}
where $\Ne$ denotes the set of positive integers.\\
In the sequel, $|\cdot|$ will always denote the underlying Euclidean or $l^2$-norm. For example
 $$|h(t,w,x,y,z)|^2=\sum_{i=1}^{+\infty} h_i(t,w,x,y,z)|^2.$$
 We define
\begin{equation*}
\begin{split}
 &f ( \cdot
,\cdot,\cdot, 0,0):=f^0 \\
& h ( \cdot,\cdot, \cdot ,0,0):=h^0 = (h^0_1,..., h^0_{i}, \cdots) \\
& g( \cdot,\cdot,\cdot ,0,0) :=g^0 = (g_1^0,...,g_d^0).\\
\end{split}
\end{equation*}
 We still consider the
quasilinear stochastic partial differential equation \eqref{e0}
   for the real-valued random  field $ u_t(x)$, that we rewrite as:
 {\small
\begin{equation}
\begin{split}
\label{e1}
  du_t (x)  =
 \left( \sum_{i,j=1}^d \partial_i a_{i,j}(t,x)\partial_j u_t (x)
+f(t,x,u_t (x),\nabla u_t (x))\right. &+ \left. \sum_{i=1}^d
\partial_i
g_i(t,x,u_t(x),\nabla u_t (x)) \right) dt\\
 &  + \sum_{i=1}^{+\infty}h_i(t,x,u_t(x),\nabla u_t(x))\, dB^i_t,\\
\end{split}
\end{equation}}
with initial condition $ u (0,.) = \xi (.) $. Let us point out that in the equation  \eqref{e1}, the divergence term $\partial_i
g_i(t,x,u_t(x),\nabla u_t (x)) $ has to be understod as
$$ \displaystyle\frac{\partial}{\partial x_i}\left( 
g_i(t,x,u_t(x),\nabla u_t (x)) \right),$$
and is defined rigorously in the weak sense (by integration by parts).
%

We also assume  that $ \xi $ is a  $\mathcal{F}_0$-measurable,
$L^{2} (\cO)$-valued random variable. We consider  the following sets of assumptions : \\[0.2cm]
 \textbf{Assumption (H)}: There exist  non negative
 constants $C,\, \alpha, \,\beta $ such that for almost all $\omega$, the following inequalities hold for all $(x,y,z,t)\in \cO\times \R\times \R^d\times \R_+$:
\begin{enumerate}
\item[\textbf{(i)}]
 $ |f(t,\omega,x,y,z) -f( t,\omega, x,y^{\p},z^{\p}) | \leq \, C\big(|
y-y^{\p}| + |z-z^{\p}| \big) $
\item[\textbf{(ii)}] $  \Big(| h(
t,\omega,x,y,z) -h( t,\omega,x,y^{\p},z^{\p})
|^2\Big)^{\frac{1}{2}}\leq C \, | y-y^{\p}|+\, \beta \, |z-z^{\p}|,$
\item[\textbf{(iii)}]
  $ \Big(\sum_{i=1}^{d}| g_i(
t,\omega,x,y,z) -g_i( t,\omega,x,y^{\p},z^{\p})
|^2\Big)^{\frac{1}{2}} \leq  C \, | y-y^{\p}|+ \, \alpha \, |
z-z^{\p}|. $
\item[\textbf{(iv)}]  the contraction property  :
$ \alpha  +\dis \frac{\beta ^{2}}{2}  < \lambda\, $.
\end{enumerate}
Moreover we  introduce some integrability conditions on  $f^0, \;
g^0, \, h^0$ and the initial data $\xi$ : \\[0.2cm]
\textbf{Assumption (HD\#)}
$$ E\left( \left( \left\| f^0\right\|_{\#;t}^*\right) ^2+\left\|
g^0\right\| _{2,2;t}^2+\left\| h^0\right\| _{2,2;t}^2\right) <\infty
, $$ for each $t\ge 0.$
\\[0.2cm]
Sometimes we shall consider the  following stronger 
conditions:\\[0.2cm]
\textbf{Assumption (HD2)}%
$$ E\left(  \left\| f^0\right\|_{2,2;t}^2+\left\|
g^0\right\| _{2,2;t}^2+\left\| h^0\right\| _{2,2;t}^2\right) <\infty
, $$ for each $t\ge 0.$
\\[0.2cm]
\textbf{Assumption (HI2)}  {\it integrability condition on the
initial condition} : $$  E \|\xi \|^2  < \infty .
$$
\begin{Rq}
Note that $\left( 2,1\right) $ is the pair of conjugates of the pair
$\left( 2,\infty \right) $ and so $\left( 2,1\right) $ belongs to
the set $I^{\prime }$ which defines the space $L_{\#;t}^*$ (see the Appendix for more details). Since
$\left\| v\right\| _{2,1;t}\le \sqrt{t}\left\| v\right\| _{2,2;t}$
for each $v\in L^{2,2}\left( \left[
0,t\right] \times {\cal O}\right) ,$ it follows that%
$$ L^{2,2}\left( \left[ 0,t\right] \times {\cal O}\right) \subset
L^{2,1;t}\subset L_{\#;t}^*, $$ and $\left\| v\right\|_{\#;t}^*\le
\sqrt{t}\left\| v\right\| _{2,2;t},$ for each $v\in L^{2,2}\left(
\left[ 0,t\right] \times {\cal O}\right) .$ This shows that the
condition \textbf{(HD\#)} is weaker than \textbf{(HD2)}.
\end{Rq}
\subsection{Main example of stochastic noise}
\label{mainex} Let $W$ be a
 noise white in time and colored in space, defined on a standard filtered probability
 space  $ \big( \Omega,\cF, (\cF_t)_{t \geq 0}, P \, \big)$ whose covariance function is given by:
 \[ \forall s,t\in\R_+ ,\ \forall x,y\in\cO,\ \ E[\dot{W} (x,s)\dot{W}(y,t)]=\delta (t-s)k(x,y),\]
 where $k:\cO\times \cO  \mapsto \R_+$ is a symmetric and measurable function.\\
 Consider the following SPDE driven by $W$:
 {\small
\begin{equation}
\begin{split}
\label{eW}
  du_t (x)  =
 \big( \sum_{i,j=1}^d \partial_i a_{i,j}(t,x)\partial_j u_t (x)
+f(t,x,u_t (x),\nabla u_t (x)) &+  \sum_{i=1}^d
\partial_i
g_i(t,x,u_t(x),\nabla u_t (x)) \big) dt\\
 &  \hspace{-1 cm}+ \tilde{h} (t,x,u_t(x),\nabla u_t(x))\, W(dt,x),\\
\end{split}
\end{equation}  }
where $f$ and $g$ are as above and $\tilde{h}$ is a random real valued function.\\
We assume that the covariance function $k$ defines a trace class operator
 denoted by $K$ in $L^2 (\mathcal{O})$. It is well known (see  \cite{RN}) that there exists an orthogonal
 basis $(e_i )_{i\in \Ne}$ of $L^2 (\cO )$ consisting of eigenfunctions of
 $K$ with corresponding eigenvalues $(\lambda_i )_{i\in\Ne}$ such that
 \[ \sum_{i=1}^{+\infty} \lambda_i <+\infty ,\]
 and
 \[ k(x,y)=\sum_{i=1}^{+\infty} \lambda_i e_i (x)e_i (y).\]
It is also well known that there exists a sequence $((B^i
(t))_{t\geq 0})_{i\in\Ne}$ of
 independent standard Brownian motions such that
 \[ W(dt, \cdot )=\sum_{i=1}^{+\infty}\lambda_i^{1/2} e_i B^i (dt).\]
 So that equation \eqref{eW} is equivalent to \eqref{e1} with $h=(h_i)_{i\in\Ne}$ where
 $$\forall i\in\Ne,\ h_i (s,x, y,z)=\sqrt{\lambda_i}\tilde{h}(s,x,y,z)e_i (x).$$
Assume as in \cite{SW} that for all $i\in\Ne$,  $\| e_i \|_{\infty} <+\infty $ and
$$\sum_{i=1}^{+\infty} \lambda_i \| e_i \|_{\infty}^2 <+\infty.$$
Since
 $$  \Big(| h(
t,\omega,x,y,z) -h( t,\omega,x,y{\p},z{\p})
|^2\Big)^{\frac{1}{2}}\leq\left( \sum_{i=1}^{+\infty} \lambda_i \| e_i \|_{\infty}^2\right) \left| \tilde{h} (t,x,y,z)-\tilde{h}(t,x,y{\p},z{\p})\right|^2,$$
$h$ satisfies the Lipschitz hypothesis {\bf (H)-(ii)} if  $\tilde{h}$ satisfies a similar  Lipschitz hypothesis.
\subsection{Spaces of processes and notion of weak solutions}
\label{Weak solutions}
We shall denote by $\PP$ the set of predictable processes which
admit a version in $L_{loc}^2\left( { %
\mathbb{R}}_{+};L^2( {\cal O}\right) )$.\\
We now introduce $\cH =\cH(\mathcal{O})$, the space of
  $H_0^1(\mathcal{O})$-valued
  predictable processes $(u_t)_{t \geq 0}$ such that
\[
\| u\|_t=\left( E \sup_{0\leq s\leq t} \left\| u_{s}\right\|
^{2}+E\int_{0}^{t}\,\|\nabla u_{s}\|^2 dt\right)
^{1/2}\;< \; \infty \;,\quad \mbox{for each} \; t>0 \, .
\]
We define $\HH_{loc}=\HH_{loc}(\mathcal{O})$ to be the set of $H^1_{loc} (\OO
)$-valued predictable processes such that for any compact subset $K$ in $\OO$
and all $t>0$:

\[
\left( E \sup_{0\leq s\leq t}\int_K | u_{s}(x)|^{2}\, dx
+E\int_{0}^{t}\int_K |\nabla u_{s}(x)|^2\, dx dt\right) ^{1/2}\;< \;
\infty .
\]
We also denote by $\hat{F}$ the subspace of elements in $\HH$ which are $L^2 (\cO )$-continuous.
Moreover, we  denote by $\HH_T$ (resp. $\hat{F}_T$ ) the set of processes which are the restrictions to $[0,T]$ of elements in $\HH$ (resp. $\hat{F}$). Let us remark that $(\hat{F}_T ,\| \cdot \|_T )$ is a Banach space.\\
The space of test functions is $\mathcal{D}=\mathcal{C}%
_{c}^{\infty } (\R_+ )\otimes \mathcal{C}_c^2 (\cO )$, where $\mathcal{C}%
_{c}^{\infty } (\R_+ )$ denotes the space of all real valued infinitely
differentiable  functions with compact support in $\mathbb{R}^+$ and
$\mathcal{C}_c^2 (\cO )$ the set
of $C^2$-functions with compact support in $\cO$.\\

\begin{Def}
We say that $ u \in \HHl $ is a weak solution of equation $\left(
\ref{e1}\right) $ with initial condition $\xi$ if the following
relation holds almost surely, for each $\varphi \in \mathcal{D},$
\begin{equation}
\label{weak2}
\begin{split}
\int_{0}^{\infty }[&\left( u_{s},\partial _{s}\varphi \right) -
\sum_{i,j=1}^d\int_{\cO}a_{i,j}(s,x)\partial_i u_{s}(x)\partial_j\varphi _{s}(x)dx +\left( f\left(
s,u_{s},\nabla u_s\right) ,\varphi _{s}\right)\\& -\sum_{i=1}^d \left(
g_i\left( s, u_{s},\nabla u_s\right) ,\partial_i \varphi _{s}\right) ]ds+\; \sum_{i=1}^{+\infty} \int_{0}^{\infty }\left( h_i\left( s, u_{s},\nabla u_s\right)
,\varphi
_{s}\right) dB^i_{s}+\left( \xi ,\varphi _{0}\right) =0.\\
\end{split}
\end{equation}
We denote by $ \mathcal{U}_{loc} (\xi,f,g,h)$ the set of
all such solutions $u$.\\

\end{Def}

 In general we do not know much about the set ${\cal U}_{loc}\left(
\xi ,f,g,h\right) $. It may be empty or may contain several
elements.  As the Sobolev space $%
H_0^1\left( {\cal O}\right) $ consists of functions which vanish  at the boundary $\partial {\cal O},$ we  say
that a solution which belongs to $\cH $ satisfies the zero Dirichlet
conditions at the boundary of ${\cal O}.$ \\
 We
denote by ${\cal U}\left( \xi ,f,g,h\right) $ the solution of
(\ref{e1}) with zero Dirichlet boundary conditions whenever it
exists and is unique, we shall prove that this is the case for example under {\bf (H)}, {\bf (HI2)} and {\bf (HD2)}.

We should also note that if the conditions \textbf{(H)},
\textbf{(HD2)} and \textbf{(HI2)} are satisfied and if $u$ is a
process in $\cH ,$ the relation from this definition holds with any
test function $\varphi \in {\cal D}$ if and
only if it holds for any test function in ${\cal C}_c^\infty \left( {\bf R}%
_{+}\right) \otimes H_0^1\left( {\cal O}\right) .$ In fact, in
this case,
one may use as space of test functions any space of the form ${\cal C}%
_c^\infty \left( {\bf R}_{+}\right) \otimes V,$ where $V$ is a
dense subspace of $H_0^1\left( {\cal O}\right) ,$ obtaining
equivalent definitions
of the notion of solution with null Dirichlet conditions at the boundary of $%
{\cal O}.$\\
Let us now  precise the sense in which a solution is dominated on
the lateral boundary. 

\begin{Def}{\label{Def3}}
If $v$ belongs to $H^1_{loc} (\OO )$, we say that
$v$ is non-positive on the boundary of $\cO$ if $v^+$ belongs to $H^1_0 (\cO)$ and denotes it simply:
$v\leq 0$ on $\partial \cO$.
\end{Def}

\section{Existence, uniqueness and estimates of the solution with null-Dirichlet condition}\label{DirichletCondition}
\subsection{Notion of mild solution}
We now turn out to the notion of {\it mild} solution:
\begin{Def} We say that $ u \in \cH $ is a mild solution of equation $\left(
\ref{e1}\right) $ with initial condition $\xi \in L^2(\Omega \times
\mathcal{O}),$ if for all $t\in \R_+$,
\begin{equation}\label{mild0} \begin{split}
u_t (x )=&\int_{\cO}G(t,x,0,y)\xi(y)\, dy
+\int_0^t\int_{\cO} G(t,x,s,y)f(s,y ,u_s (y),\nabla u_s (y))dy
ds \\
&\ \ +\sum_{i=1}^{d}\int_0^t\int_{\cO}
G(t,x,s,y)\partial_{i}g_i(s,. ,u_s ,\nabla u_s )(y)dy ds
\\&\ \ +\sum_{i=1}^{+\infty}\int_0^t\int_{\cO} G(t,x,s,y)h_i(s,y ,u_s (y),\nabla
u_s (y))dB^i_s .\\
\end{split}
\end{equation}
\end{Def}
Let us remark that thanks to Gaussian estimate (\ref{Gaussestimate}), all the
quantities in (\ref{mild0}) are well defined excepted the term
\[\int_0^t\int_{\cO}
G(t,x,s,y)\partial_{i,y}g_i(s,. ,u_s ,\nabla u_s )(y)dy ds.\]
This last term has to be understood in the weak sense  thanks to the following
 Proposition:
\begin{Pro}{\label{opU}} Let $U: \left( C_c^{\infty} (\R_+
)\otimes H^1_0 (\cO )\right)^d \longrightarrow \hat{F}$ be defined by
\[\forall \tilde{w}\in \left( C_c^{\infty} (\R_+
)\otimes H^1_0 (\cO )\right)^d ,\  \forall t\geq 0 ,\ (U \tilde{w})_t =\sum_{i=1}^d \int_0^t G(t,\cdot ,s
,y)\partial_{i,y}\tilde{w}_{i,s} (y) dy ds .\]The operator $U$ admits a uniquely determined
continuous extension
\[ U: L^2 _{loc} (\R_+ ;L^2 ( \cO )^d )\longrightarrow \hat{F},\]
that we still denote \[\forall t\geq 0 ,\ (U\tilde{w})_t
=\sum_{i=1}^d\int_0^t G(t,\cdot ,s ,y)\partial_{i,y}\tilde{w}_{i ,s}(y) dy ds
.\]Moreover if $\tilde{w}\in L^2 _{loc} (\R_+ ;L^2 ( \cO )^d )$, $u=U\tilde{w}$ is
the weak solution of
\[ du_t (x) =
 \sum_{i,j=1}^d \partial_i a_{i,j}(t,x)\partial_j u_t
 (x)+\sum_{i=1}^d \partial_i \tilde{w}_i \ ,\ u_0 =0 ,\]
and it satisfies the following relation:
\begin{equation}\label{Append1} \frac 12 \|u_t \|^2 +\int_0^t \sum_{i,j=1}^d\int_{\cO}
a_{i,j}(s,x)\partial_i u_s (x)\partial_j u_s (x)\, dx\, ds
=-\sum_{i=1}^d \int_0^t \left(  \tilde{w}_{i,s} ,\partial_iu_s \right)\,
ds, \ \ t\geq 0.\end{equation}
As a consequence, we have the following
estimate:
\begin{equation}\label{Append2}\| u\|_T^2\leq C_{\lambda}\int_0^T \| \tilde{w}_s \|^2 \, ds ,\end{equation}
where $C_{\lambda}$ is a constant depending only on $\lambda$.
\end{Pro}
\begin{proof} See Subsection \ref{proofopU} in the Appendix.
\end{proof}

\subsection{The linear case}\label{linearcase}
Let  $\xi\in L^{2}(\Omega ,\mathcal{F}_{0},P;L^2 (\cO )),
\,w=(w_i )_{i\in\Ne}\in \PP^{\Ne}, w' \in \PP, w''\in \PP^d  .$ We assume that
$$ | w|=\left({ \sum_{i=1}^{+\infty} |w_i^2 |}\right)^{1/2}\in\PP.$$We set
\begin{equation}\begin{split}
u_t (\cdot )=&\int_{\cO}G(t,\cdot,0,y)\xi(y)\, dy
+\int_0^t\int_{\cO} G(t,\cdot,s,y)w'_s (y)dy
ds \\
&\ \ +\sum_{i=1}^{d}\int_0^t\int_{\cO}
G(t,\cdot,s,y)\partial_{i}w''_i (y)dy ds
\\&\ \ +\sum_{i=1}^{+\infty} \int_0^t\int_{\cO} G(t,\cdot,s,y)w_{i,s} (y) dB^i_s.\\
\end{split}
\end{equation}
The goal of this section is to prove that $u$ is the unique solution of the {\it linear} equation
\begin{equation}{\label{linear0}}
\begin{split}
  du_t (x)  =
\Big( \sum_{i=1}^d \partial_i &\big[\sum_{j=1}^d a_{i,j}(t,x)\partial_j u_t (x)
+w''_{i,s} (x) \big]+w'_s (x)
\Big) dt + \sum_{i=1}^{+\infty}w_{i,s}(x)\, dB^i_t,\\
\end{split}
\end{equation}  with initial condition $u_0=\xi$ and zero Dirichlet condition on the boundary:
$$u(t,x)=0 \ \ \forall (t,x)\in (0,+\infty)\times\partial\cO.$$ To this end we proceed as follows: first we prove the result in the case where all the coefficients are regular and then, using an approximation procedure, we prove it in the general case. This second part is quite long and we shall split the proof in several steps.
\subsubsection{The regular case}
We assume first that all the coefficients are regular and that $\partial \cO $ is smooth. In this case, existence and uniqueness are well known (see for example \cite{KR}), nevertheless we give the proof 
in order to explicit the estimates we  need to pass to the limit in the general case.
\begin{Pro}
\label{Ito1} Assume that $\partial \mathcal{O}$ is smooth,  all the coefficients $ a_{i,j}$ belong to $ C^{\infty} \left( \R_+ \times  \overline{\mathcal{O}} \right)$, $\xi\in C^{\infty}_c (\cO )$,
$\,w, w' \in \left(L^2 (\Omega )\otimes C_c ([0,+\infty
))\otimes C_c^{\infty}({\cO})\right)\bigcap \PP$  and \\  $w''\in \left(L^2
(\Omega )\otimes C_c ([0,+\infty ))\otimes
C_c^{\infty}({\cO})\right)^d \bigcap \PP^d  .$ We set
\begin{equation}\begin{split}
u_t (\cdot )=&\int_{\cO}G(t,\cdot,0,y)\xi(y)\, dy
+\int_0^t\int_{\cO} G(t,\cdot,s,y)w'_s (y)dy
ds \\
&\ \ +\sum_{i=1}^{d}\int_0^t\int_{\cO}
G(t,\cdot,s,y)\partial_{i,y}w''_i (y)dy ds
\\&\ \  +\sum_{i=1}^{+\infty} \int_0^t\int_{\cO} G(t,\cdot,s,y)w_{i,s} (y) dB^i_s\, .\\
\end{split}
\end{equation}

Then $u$ has a version in $\hat{F} $ and is the unique  solution in the weak sense of $\eqref{linear0}$ in $\cH$ i.e. the unique element in $\cH$ such that  for each $\varphi \in
\mathcal{D}$,   the following relation holds 
almost surely:
\begin{equation}
\label{weak0}
\begin{split}
\left( \xi ,\varphi _{0}\right)+\int_{0}^{\infty }[\left(
u_{s},\partial _{s}\varphi \right) - \sum_{i,j=1}^d \int_{\cO}
a_{i,j}(s,x)\partial_i u_{s}(x)\partial_j \varphi _{s}(x)dx - \left(
w'_s ,\varphi _{s}\right)]ds
\\+\int_0^{+\infty}\sum_{i=1}^d \left( w''_{i,s} ,\partial_i \varphi _{s}\right) ds +\;
\sum_{i=1}^{+\infty} \int_{0}^{\infty }\left( w_{i,s}, \varphi
_{s}\right) dB^i_s =0.\\
\end{split}
\end{equation}
Moreover, we have the following estimates
for all $t\geq 0$:
\begin{equation}
\label{Ito0}
\begin{split}
& \left\| u_{t}\right\| ^{2}+2\int_{0}^{t}\sum_{i,j}\int_{\cO}a_{i,j}
(s,x)\partial_i u_s (x)\partial_j u_s (x)dxds=\left\| \xi \right\|
^{2}+2 \sum_i^{+\infty} \int_0^t (w_{i,s}, u_s ) \,dB^i_s\\\
&
+2\int_{0}^{t}\left( u_{s},w_{s}^{\prime }\right) ds-2\sum_i
\int_{0}^{t}\left( \partial_i u_s ,w''_{i,s}\right)
ds+\int_0^t\| |w_s |\|^2 \, ds,\\
\end{split}
\end{equation}

and
\begin{equation}
\label{priori:estimate}
E[\left\| u\right\| _{T}^{2}]\leq cE\left[ \left\| \xi\right\|
^{2}+\int_{0}^{T}\left( \left\| | w_{t}|\right\|^{2}+\left\|
w_{t}^{\prime }\right\| ^{2}+\sum_i \left\| w_{i,t}^{\prime \prime
}\right\| ^{2}\right) dt\right]
\end{equation}

where $c$ is a constant which only depends on $T.$

\end{Pro}
\begin{proof}
Following Aronson \cite{Aronson3}, we know that the weak fundamental
solution $G$ is a classical one. Moreover, it is well known that
$G$ is one time differentiable with respect to time and infinitely differentiable with respect to space
variables in $\Delta$ and that we have the following estimate for all
$(t,x,s,y)\in\Delta$ and $1\leq i,j\leq d$ (see
\cite{EidIva} for example):
\begin{eqnarray}{\label{estimDerivx}}
| \partial_{i,x}^l\partial_{j,y}^k G(t,x,s,y)|&\leq &C
(t-s)^{-\frac{d+l+k}{2}}\exp(-\varrho\frac{|x-y|^2}{(t-s)}),\\{\label{estimDerivt}}
| \partial_t G(t,x,s,y)|&\leq &C
(t-s)^{-\frac{d}{2}-1}\exp(-\varrho\frac{|x-y|^2}{(t-s)})
\end{eqnarray}
with $k,l=0$, $1$ or $2$ and where $\partial^l_{i,x}$ denotes the partial derivative of order $l$ with respect to the variable $x_i$.\\
 Due to the regularity of $G$ and of all the coefficients in the expression of $u$, one can use
 the fact that $G$ is a strong solution. As a consequence, $u  $ is a $H_0^2 (\cO )$-valued semi-martingale with $L^2 (\cO )$-continuous trajectories   (see \cite{KR}, Chapter 1 or \cite{DZ}) and we have the following integral representation:
 \begin{equation}{\label{SemMartin}}\begin{split}
 u_t (x)=&\xi(x)+\sum_{i,j=1}^d \int_0^t \partial_i \left( a_{i,j}(s,x)\partial_j u_s (x)\right) ds+\int_0^tw'_s (x)ds
 +\sum_{i=1}^{d}\int_0^t \partial_{i,y}
w''_{i,s} (x) ds
\\&\ \ +\sum_{i=1}^{+\infty}\int_0^t w_{i,s} (x) dB^i_s.\\
\end{split}
\end{equation}
Applying the It\^{o}'s
formula for Hilbert-valued semimartingaled (see \cite{KR} Chapter 1, Section 3) and then integrating with respect to $x$, we get
\begin{equation}\label{repre1}\begin{split}
\parallel u_t\|^2 &+2\int_0^t \int_{\cO} a_{i,j} (s,x)\partial_i u_s(x)\partial_j u_s(x)dxds=
\| \xi\|^2 +2\int_0^t \left( w'_s,
u_s\right)ds\\&+2\sum_i \int_0^t\left(
w''_{i,s} ,\partial_i v^{n,m}_s\right)ds
 +2\sum_i \int_0^t \left( w_{i,s}  ,u_s \right)dB^i_s\\&
+\int_0^t \| |w_s |\parallel^2
ds.
\end{split}\end{equation}
Fix $\varepsilon
>0$ small. We have for all $t\in [0,T)$:
\begin{eqnarray*}
2\left| \int_0^t \left( w'_s,
u_s\right)ds\right| &\leq& \varepsilon \int_0^T \| u_s
\|^2 ds +\frac{1}{\varepsilon} \int_0^T \| w'_s
\|^2 ds
\end{eqnarray*}
and
\begin{eqnarray*}
2\left| \sum_i \int_0^t\left( w''_{i,s}
,\partial_i u_s\right)ds\right| &\leq& \varepsilon \int_0^T
 \| \nabla u_s \|^2 ds+\frac{1}{\varepsilon}
\int_0^T\| |w''_{s}|\|^2 ds.
\end{eqnarray*}
 Moreover, thanks to the Burkholder-Davies-Gundy
inequality, we get
\begin{eqnarray*}
\begin{split}
E[\sup_{t\in [0,T]}\left|\sum_i \int_0^t \left(
w_{i,s}  ,u_s\right)dB^i_s\right|]& \ \  \leq
c_1 E\left[\left( \int_{0}^{T}\sum_{i=1}^{+\infty}
\left( w_{i,s} ,u_s\right)^2
ds\right)^{1/2}\right]\\& \ \ \leq c_1 E\left[\left(
\int_{0}^{T}\sum_{i=1}^{+\infty} \sup_{t\in [0,T]}\left\|
u_t \right\|^2 \| w_{i,s} \|^2 dt
\right)^{1/2}\right]\\& \  \ \ \leq c_1 E\left[\sup_{t\in
[0,T]}\left\| u_t\right\|
\left(\int_{0}^{T}\| |w_s |
\|^2dt \right)^{1/2}\right]\\& \ \ \  \leq \eps
E\left[\sup_{t\in [0,T]}\left\| u_t\right\|^2\right]
+\frac{c_1 }{4\eps}
E\left[\int_{0}^{T}\| |w_s |\|^2dt \right].
\end{split}
\end{eqnarray*}
Then using the ellipticity assumption on $a$ and the inequalities
above, by taking   the supremun in $t\in [0,T]$ in relation
\eqref{repre1} and then the expectation, we get:
\begin{equation}\label{repre}\begin{split}
&(1-2\eps (T+1))E[\sup_{t\in [0,T]}\parallel
u_t\|^2]+(2\lambda -\eps)E\int_0^t \|\nabla
u_s\parallel^2 ds\leq  2\| \xi \|^2
\\&+\frac{2}{\varepsilon} E\int_0^T \| w'_s \|^2
ds+\frac{2}{\varepsilon} E\int_0^T\|
|w''_{s}|\|^2 ds+\frac{c_1
}{2\eps} E\left[\int_{0}^{T}\| |w_s
|\|^2dt \right].
\end{split}\end{equation}

Taking $\eps$ small enough, we deduce   that we have the following a
priori estimate:
\begin{equation}
\label{estimateup} E\left\| u\right\| _{T}^{2}\leq cE\left(
\left\| \xi\right\| ^{2}+\int_{0}^{T}\left( \left\|
|w_{t}|\right\|^{2}+\left\| w_{t}^{\prime }\right\| ^{2}+\sum_i
\left\| w_{i,t}^{\prime \prime }\right\| ^{2}\right) dt\right)
\end{equation}
where $c$ is a constant which only depends on $T$ and $\lambda$
 but not on $\cO$. This proves inequality $\eqref{priori:estimate}$.\\
Relation $\eqref{weak0}$ and the fact that $u$ is a weak solution are direct consequences of  It\^{o}'s formula.\\
Finally, uniqueness is clear, indeed if $v$ is another element in
$\HH\bigcap L_{loc}^2 (\R_+; L^2 (\Omega;H_0^1 (\cO) $ which
satisfies (\ref{weak0}) for all $\vphi\in C_c^{\infty}
([0,+\infty[)\otimes C_c^{\infty} (\cO )$, then $\zeta=u-v$
satisfies
\begin{equation}
\int_{0}^{\infty }[\left(\zeta_{s},\partial _{s}\varphi \right) -
\sum_{i,j}\int_{\cO} a_{i,j}(s,x)\partial_i \zeta_{s}(x)\partial_j
\varphi _{s}(x)dx ]ds =0,
\end{equation}
standard results on deterministic PDE's ensure that $\zeta=0$.

\end{proof}
\subsubsection{The general case}
Here, we only assume that $a$ is measurable and satisfies assumption $\eqref{ellip}$, that $\cO$ is a bounded open domain without any condition on its boundary and we are given coefficients: $\xi\in L^{2}(\Omega ,\mathcal{F}_{0},P;L^2 (\cO )),
\,  w' \in \PP, \, w''=(w''_1 ,\cdots ,w''_d )\in \PP^d $ and $w=(w_i )_{i\in\Ne}\in \PP^{\Ne},$ such that $E[\int_0^T \| |w_s |\|^2 \, ds]<+\infty$. \\
We first prove that Proposition \ref{Ito1} remains true in this case and then we establish  It\^{o}'s formula for the solution. To do that, we approximate the coefficients, the domain and the second order operator in the following way:
\begin{enumerate}
\item We mollify coefficients $a_{i,j}$ and so consider
sequences $(a^n_{i,j})_n$ of $C^{\infty}$ functions such that for
all $n\in\Ne$, the matrix $a^n$ satisfies the same ellipticity and
boundedness assumptions as $a$ and
\[ \forall 1\leq i,j\leq d ,\ \lim_{n\rightarrow +\infty}
a^n_{i,j}= a_{i,j}\   a.e.\]
\item We  approximate $\cO$ by  an increasing sequence of smooth domains $(\cO^n )_{n\geq 1}$.
\item We consider  a sequence $(\xi^n )$ in $C_c^{\infty}
({\cO})$ which converges to $\xi$ in $L^2 (\cO )$ and such that for all $n$, ${\rm supp\,} \xi^n \subset \cO^n$.
\item For each $i\in\Ne$, we construct a sequence $(w^n_i
)$  in $\left(L^2 (\Omega )\otimes C_c ([0,+\infty
))\otimes C_c^{\infty}({\cO})\right)\bigcap \PP$ which converges in
$L^2_{loc} (\R_+ ;L^2 (\Omega \times \cO))$ to $w_i$ such that for all $n$, ${\rm supp\,} w^n_i \subset \cO^n$ and
$$\forall t\geq 0 ,\ E[ \int_0^t \| w^n_{i,s}\|^2 \, ds] \leq E[ \int_0^t \| w_{i,s}\|^2 \, ds],$$
so that
$$E[\int_0^t \| |w^n_s|\|^2 ds ]\leq E[\int_0^t \| |w_s|\|^2 ds ]<+\infty.$$
\item We consider a sequence  $(w'^{,n})$  in $\left(L^2 (\Omega )\otimes C_c ([0,+\infty
))\otimes C_c^{\infty}({\cO})\right)\bigcap \PP$ which converges in
$L^2_{loc} (\R_+ ;L^2 (\Omega \times \cO))$ to  $w'$
and such that for all $n$, ${\rm supp \,} w'^{,n} \subset \cO^n$.
\item Finally,  let $(w''^{,n})$ be a sequence in $\left(L^2
(\Omega )\otimes C_c ([0,+\infty ))\otimes
C_c^{\infty}({\cO})\right)^d \bigcap \PP^d $ which converges in
$L^2_{loc} (\R_+ ;L^2 (\Omega \times \cO)^d)$ to $w''$and such that for all $n$, ${\rm supp\,} w''^{,n} \subset \cO^n$ .
\end{enumerate}
For all
$n\in\Ne$, we put $\Delta^n  =\{(t,x,s,y)\in\R_+ \times \cO^n\times \R_+ \times \cO^n; t> s\}.$ We denote  by $G^n :\Delta^n \mapsto \R_+$ the
  weak fundamental solution of the problem \eqref{weak:fundamental} associated to $a^n$ and
  $\cO^n$:
\begin{equation}
\partial_t G^n (t,x;s,y)  -
\sum_{i,j=1}^d \partial_i a^n_{i,j}(t,x)\partial_j G^n (t,x;s,y) =0
\end{equation}
with  Dirichlet boundary condition $ G^n(t,x,s,y) = 0 , \quad
\mbox{for all} \; (t,x) \in \; (s,\, +\infty ) \times \partial
\cO^n
\,$.\\
In a natural way we extend $G^n$ on $\Delta$ by setting:
$G^n\equiv0$ on $\Delta \setminus \Delta^n$.\\
We define the process $u^n$ by setting for all
$(t,x)\in \R_+ \times \cO$:
\begin{equation}
\label{approximation}
\begin{split}
u_t^n (x)=&\int_{\cO}G^n (t,x,0,y)\xi^n(y) dy+\int_0^t\int_{\cO} G^n (t,x,s,y)w'^{,n}_s (y)dy
ds \\
&\ \ -\sum_{i=1}^{d}\int_0^t\int_{\cO} \partial_{i,y}
G^n (t,x,s,y)w''^{,n}_{i,s} (y)dy ds
\\&\ \ +\sum_{i=1}^{+\infty}\int_0^t\int_{\cO} G^n (t,x,s,y)w^n_{i,s} (y) \, dy dB^i_s \, .\\
\end{split}
\end{equation}
The key Lemma is the following:
\begin{Le}{\label{convGreen}}
 There exists a subsequence of $(G^n )_{n\geq 1}$ which converges everywhere to $G$  on $\Delta $, where $G$ still denotes the fundamental solution of  \eqref{weak:fundamental}.
\end{Le}
\begin{proof} Let $K$ be a compact subset of $\Delta$. There exists $\varepsilon >0$, $\eta >0$ such that  for $ (t,x) \in K$,  $ |t-s| \geq \eta$, $ d (x, \partial \mathcal{O}_n) \geq \varepsilon $ and $ d (y, \partial \mathcal{O}_n) \geq \varepsilon $, for $n$ large enough. Then using Theorem C in Aronson (\cite{Aronson3} p.616) we know that the sequence of functions $(G^n )_n$ is equicontinuous  on $K$. Moreover thanks to the  Gaussian estimates  \eqref{Gaussestimate} and Ascoli theorem, we have that $(G^n )_n$ converges uniformly to $G$ on $K$, for some subsequence. We conclude by taking an exhaustive sequence of compact subsets in $\Delta$ and a diagonalisation procedure.
\end{proof}
For simplicity, from now on we assume that the sequence $(G^n )$ is chosen such that it converges  to $G$ on $\Delta$.
\begin{Th}{\label{MainTheo} } Assume that the general hypotheses of Subsection \ref{hypo} hold. \\Let $\xi\in L^{2}(\Omega ,\mathcal{F}_{0},P;L^2 (\cO )),
\,  w' \in \PP, \, w''=(w''_1 ,\cdots ,w''_d )\in \PP^d $ and $w=(w_i )_{i\in\Ne}\in \PP^{\Ne},$ such that $E[\int_0^T \| |w_s |\|^2 \, ds]<+\infty$. We set
\begin{equation}\label{expressionofu}\begin{split}
u_t (\cdot )=&\int_{\cO}G(t,\cdot,0,y)\xi(y)\, dy
+\int_0^t\int_{\cO} G(t,\cdot,s,y)w'_s (y)\, dy
ds \\
&\ \ +\sum_{i=1}^{d}\int_0^t\int_{\cO}
G(t,\cdot,s,y)\partial_{i}w''_i (y)\, dy ds
\\&\ \ +\sum_{i=1}^{+\infty}\int_0^t\int_{\cO} G(t,\cdot,s,y)w_{i,s} (y) \, dy dB^i_s,\\
\end{split}
\end{equation}
then all the results of Proposition \ref{Ito1}  remain valid.\\
\end{Th}
\begin{proof} Let us first note that \eqref{expressionofu} is well defined thanks to Proposition \eqref{opU}. As $\cO^n$ and $a^n$ are smooth, hypotheses of the previous
subsection are fulfilled so that for all $n$, $u^n$ satisfies
Propositions \ref{Ito1} and Proposition \ref{Ito2}  with domain $\cO^n$, operator given
by  $a^n$ and coefficients $w^n$, $w'^{,n}$ and $w''^{,n}$ . Moreover, we know that for all $n\in\Ne$, the restriction of $u^n$ to $\cO^n$ belongs to $L^2_{loc}
(\R_+ ;H_0^1 (\cO^n ))$ and is a $H_0^2 (\cO^n )$-valued semimartingale hence as we put $u^n_t \equiv 0$ on $\Delta
\setminus \Delta^n$, $u^n$ belongs to $\cH$ and is a $H_0^2 (\cO  )$-valued semimartingale which admits the following decomposition:
 \begin{equation}{\label{SemMartin}}\begin{split}
 u_t^n (x)=&\xi^n(x)+\sum_{i,j}\int_0^t \partial_i \left( a^n_{i,j}(s,x)\partial_j u^n_s (x)\right) ds+\int_0^tw'^{,n}_s (x)ds
 -\sum_{i=1}^{d}\int_0^t \partial_{i}
w''^{,n}_{i,s} (x) ds
\\&\ \ +\sum_{i=1}^{+\infty}\int_0^t w^n_{i,s} (x)  dB^i_s.\\
\end{split}
\end{equation}
Let
us now pass to the limit in $\cH$.
For simplicity, we work on the finite time-interval $[0,T]$ and consider the Hilbert space
\[ F_T =L^2 (\Omega \times [0,T]; H_0^1 (\cO )),\]
equipped with the norm
\[\| u\|_{F_T }=
\left( \, E\int_0^T \left\| u_{t}\right\| ^{2}dt+\int_{0}^{T}E\,
\left\| \nabla u_t \right\|^2  dt\right) ^{1/2}\, .
\]
 Estimate \eqref{priori:estimate} ensures that the sequence $(u^n )$ is bounded in $F_T$, that for all $i\in \{ 1,\cdots ,d \}$, as the $a^n$ are uniformly bounded by $M$, the sequence $(\sum_{j=1}^d a^n_{i,j}\partial_j u^n )$ is bounded in $L^2 (\Omega\times [0,T]\times \cO )$. As a consequence,   we can extract (successively) a subsequence $(u^{n_k})_{k\geq 1}$ which converges weakly in $F_T$ to an element $\tu$ and such that for each $i$, the sequence $(\sum_{j=1}^d a^{n_k}_{i,j}\partial_j u^{n_k} )_k$ converges weakly in $L^2 (\Omega\times [0,T]\times \cO )$ to an element $v_i$. Since clearly $(\sum_{j=1}^d (a^{n_k}_{i,j}-a_{i,j})\partial_j u^{n_k} )_k$ converges to $0$ in $L^2 (\Omega\times [0,T]\times \cO )$ we conclude that $v_i= \sum_{j=1}^d a_{i,j}\partial_j \tu $. 
 Therefore, we can construct a sequence $(\tu^n )$ of convex combinations of elements in $(u^{n_k} )$ of the form
 \[ \tu^n =\sum_{k=1}^{N_n } \alpha^n_k u^{n_k}\]
 with $\lim_{n\rightarrow +\infty}N_n =+\infty$,  $\alpha^n_k \geq 0$, $\sum_{k=1}^{N_n } \alpha^n_k=1$ for all $n$ and
 such that:
 \begin{enumerate}
 \item $(\tu^n )$ converges strongly in $F_T$ to an element $\tu\in F_T$,
 \item $\forall i\in \{ 1, \cdots ,d\}$, $\tA^n_i =\sum_{j=1}^d \sum_{k=1}^{N_n} \alpha^n_k a_{i,j}^{n_k}\partial_j u^{n_k}$ converges to  $\sum_{j=1}^d a_{i,j} \partial_j \tu $  in $L^2 (\Omega\times[0,T]\times \cO )$ as $n$ goes to infinity.
 \end{enumerate}
  
 In a natural way, we set:
\[
 \txi^n =\sum_{k=1}^{N_n } \alpha^n_k \xi^{n_k},\, \tw'^{,n} = \sum_{k=1}^{N_n } \alpha^n_k w'^{,n_k},\, \tw''^{,n} = \sum_{k=1}^{N_n } \alpha^n_k w''^{,n_k}, \tw^n_i = \sum_{k=1}^{N_n } \alpha^n_k w^{n_k}_i \, \forall i\geq 1.\]
So $\tu^n$ admits the following representation for all $n\geq 1$:
  \begin{equation}{\label{SemMartin2}}\begin{split}
 \tu_t^n (x)=&\txi^n(x)+\sum_{i=1}^d\int_0^t \partial_i \tA^n_i (s,x) ds+\int_0^t \tw'^{,n}_s (x)ds
 -\sum_{i=1}^{d}\int_0^t \partial_{i,y}
\tw''^{,n}_{i,s} (x) ds
\\&\ \ +\sum_{i=1}^{+\infty}\int_0^t \tw_{i,s}^n (x)  dB^i_s\,.\\
\end{split}
\end{equation}
Let $n,m \in\Ne$, we set $v^{n,m}=\tu^n -\tu^m$. Applying It\^o's
formula and then integrating with respect to $x$, we get
\begin{equation}\label{repre}\begin{split}
\parallel v_t^{n,m}\|^2 =&\| \txi^n -\txi^m \|^2 -2\sum_i\int_0^t \int_{\cO} \left( \tA^n_i (s,x)-\tA^m_{i} (s,x)\right)\partial_i v^{n,m}_s(x)dxds
\\&+2\int_0^t \left( \tw'^{,n}_s-\tw'^{,m}_s,
v^{n,m}_s\right)ds+2\sum_i \int_0^t\left(
\tw''^{,n}_{i,s}-\tw''^{,m}_{i,s} ,\partial_i v^{n,m}_s\right)ds\\&
 +2\sum_i \int_0^t \left( \tw^n_{i,s} -\tw^m_{i,s} ,v^{n,m}_s\right)dB^i_s
+\sum_i \int_0^t \| \tw^n_{i,s}  -\tw^m_{i,s}  \parallel^2
ds.
\end{split}\end{equation}
Let $\vep >0$,
we have for almost all $t\in [0,T)$ and $x\in \cO$:
\begin{eqnarray*}
\sum_i \left( \tA^n_i (t,x)-\tA^m_{i} (t,x)\right)\partial_i v^{n,m}_t(x)&=&\sum_{i,j}a_{i,j}(t,x)\partial_j v^{n,m}_t(x) \partial_i v^{n,m}_t (x)\\
&&+\sum_i \left( \tA^n_i (t,x)-\sum_j a_{i,j} (t,x)\partial_j \tu^n_t (x)\right) \partial_i v^{n,m}_t (x)\\
&&-\sum_i \left( \tA^m_i (t,x)-\sum_j a_{i,j} (t,x)\partial_j \tu^m_t (x)\right) \partial_i v^{n,m}_t (x)
\end{eqnarray*}
this yields, thanks to the ellipticity asumption on the matrix $a$:
\begin{eqnarray*}\begin{split}
-2\sum_i\int_0^t \int_{\cO} &\left( \tA^n_i (s,x)-\tA^m_{i} (s,x)\right)\partial_i v^{n,m}_s(x)dxds\leq -2\lambda \int_0^t \|\nabla v^{n,m}_s\|^2 ds\\
&+ 2\sum_i \int_0^t\int_{\cO} | \tA^n_i (s,x)-\sum_j a_{i,j} (s,x)\partial_j \tu^n_s (x)| \left|\partial_i v^{n,m}_s (x)\right|\, ds dx\\
&\ \ +2\ \sum_i \int_0^t\int_{\cO}| \tA^m_i (s,x)-\sum_j a_{i,j} (s,x)\partial_j \tu^m_s (x)| \left|\partial_i v^{n,m}_s (x)\right|\, ds dx\, .\end{split}
\end{eqnarray*}
Using the trivial inequality $2ab \leq \vep a^2 +\disp\frac1\vep b^2$ we get
\begin{eqnarray*}
2\left| \int_0^t \left( \tw'^{,n}_s-\tw'^{,m}_s,
v^{n,m}_s\right)ds\right| &\leq& \varepsilon \int_0^T \| v^{n,m}_s
\|^2 ds +\frac{1}{\varepsilon} \int_0^T \| \tw'^{,n}_s-\tw'^{,m}_s
\|^2 ds
\end{eqnarray*}
and
\begin{eqnarray*}
2\left| \sum_i \int_0^t\left( \tw''^{,n}_{i,s}-\tw''^{,m}_{i,s}
,\partial_i v^{n,m}_s\right)ds\right| &\leq& \varepsilon \int_0^T
 \| \nabla v^{n,m}_s \|^2 ds+\frac{1}{\varepsilon}
\int_0^T\| |\tw''^{,n}_{s}-\tw''^{,m}_{s}|\|^2 ds \, .
\end{eqnarray*}
 Moreover, thanks to the Burkholder-Davies-Gundy, we obtain
\begin{eqnarray*}
\begin{split}
&E[\sup_{t\in [0,T]}\left|\sum_i \int_0^t \left(
\tw^n_{i,s} -\tw^m_{i,s} ,v^{n,m}_s\right)dB^i_s\right|]\\&\  \leq
c_1 E\left[\left( \int_{0}^{T}\sum_{i=1}^{+\infty}
\left( \tw^n_{i,s} -\tw^m_{i,s}  ,v^{n,m}_s\right)^2
ds\right)^{1/2}\right]\\& \ \ \leq c_1 E\left[\left(
\int_{0}^{T}\sum_{i=1}^{+\infty} \sup_{t\in [0,T]}\left\|
v^{n,m}_t \right\|^2 \| \tw^n_{i,s} -\tw^m_{i,s} \|^2 dt
\right)^{1/2}\right]\\& \  \ \ \leq c_1 E\left[\sup_{t\in
[0,T]}\left\| v^{n,m}_t\right\|
\left(\int_{0}^{T}\| |\tw^n_{s} -\tw^m_{s} |
\|^2dt \right)^{1/2}\right]\\& \ \ \ \ \leq \eps
E\left[\sup_{t\in [0,T]}\left\| v^{n,m}_t\right\|^2\right]
+\frac{c_1 }{4\eps}
E\left[\int_{0}^{T}\|| \tw^n_{s} -\tw^m_{s}| \|^2dt \right].
\end{split}
\end{eqnarray*}
Then using  the inequalities
above, by taking   the supremun in $t\in [0,T]$ in relation
\eqref{repre} and then the expectation, we get:
\begin{equation}\label{cauchy}\begin{split}
&(1-\eps (T+2))E[\sup_{t\in [0,T]}\parallel
v_t^{n,m}\|^2]+(2\lambda -\eps)E\int_0^T \|\nabla
v^{n,m}_s\parallel^2 ds\leq E[ \| \xi^n -\xi^m \|^2]
\\&+ 2\sum_i E\left[ \int_0^T \int_{\cO} | \tA^n_i (s,x)-\sum_j a_{i,j} (s,x)\partial_j \tu^n_s (x)| \left|\partial_i v^{n,m}_s (x)\right|\, ds dx\right]\\
&\ +2\ \sum_i E\left[\int_0^T \int_{\cO}| \tA^m_i (s,x)-\sum_j a_{i,j} (s,x)\partial_j \tu^m_s (x)| \left|\partial_i v^{n,m}_s (x)\right|\, ds dx\right]
\\&\ \ +\frac{1}{\varepsilon} E\left[\int_0^T \| \tw'^{,n}_s-\tw'^{,m}_s \|^2
ds\right] +\frac{1}{\varepsilon} E\left[\int_0^T\|
|\tw''^{,n}_{s}-\tw''^{,m}_{s}|\|^2 ds\right]\\&\ \ \ +\frac{c_1
}{4\eps}E\left[\int_{0}^{T}\| |\tw^n_s
 -\tw^m_s | \|^2ds \right].
\end{split}\end{equation}
Let us prove now that each term in the right member tends to $0$ as $n,m$ go to $+\infty$.\\
First of all, by construction of the approximating sequences $(\xi^n )_n$, $(w^n )_n$, $(w^{',n})_n$ and $(w^{'',n})_n$ given at the beginning of this step, we have
\[ \hspace{-0.5cm}\lim_{n,m\rightarrow +\infty} E[ \|\txi^n -\txi^m \|^2] = \lim_{n,m\rightarrow +\infty} E\left[\int_0^T \| \tw'^{,n}_s-\tw'^{,m}_s \|^2
ds\right]= \lim_{n,m\rightarrow +\infty}E\left[\int_0^T\|
|\tw''^{,n}_{s}-\tw''^{,m}_{s}|\|^2 ds\right] =0, \]
and
$$\lim_{n\rightarrow + \infty}E\left[ \int_0^T \| | \tw^n_s -w_s |\|^2 \, ds\right]=0.$$
Let $i\in \{1,\cdots ,d\}$. As $a$ is bounded, $\sum_j a_{i,j} \partial_j \tu^n$ tends to $\sum_j a_{i,j}\partial_j \tu$ and so  $ | \tA^n_i -\sum_j a_{i,j} \partial_j \tu^n_s (x)| $ tends to $0$ in $L^2 (\Omega\times [0,T]\times \cO)$ as $n$ goes to $+\infty$. As $(v^{n,m})_{n,m}$ is bounded in $L^2 (\Omega\times [0,T]\times \cO)$, we deduce from this that
\[\lim_{n,m\rightarrow +\infty} E\left[ \int_0^T \int_{\cO} | \tA^n_i (s,x)-\sum_j a_{i,j} (s,x)\partial_j \tu^n_s (x)| \left|\partial_i v^{n,m}_s (x)\right|\, ds dx\right]=0,\]
and in the same way
\[\lim_{n,m\rightarrow +\infty} E\left[ \int_0^T \int_{\cO} | \tA^m_i (s,x)-\sum_j a_{i,j} (s,x)\partial_j \tu^m_s (x)| \left|\partial_i v^{n,m}_s (x)\right|\, ds dx\right]=0.\]
Taking $\vep$ small enough in \eqref{cauchy}, we conclude that $(\tu^n )$ is a Cauchy sequence in $\hat{F}_T$ it is clear that its limit is $\tu$ so we have
\[ \lim_{n\rightarrow +\infty} \| \tu^n -\tu \|_T =0.\]
It remains to prove that $\tu=u$.\\
 We have for all $n$:
\begin{equation}{\label{approx2}}
\begin{split}
u_t^{n} (x)=&\int_{\cO}G^n (t,x,0,y)\xi^n(y) dy+\int_0^t\int_{\cO} G^n (t,x,s,y)w'^{,n}_s (y)dy
ds \\
&\ \ -\sum_{i=1}^{d}\int_0^t\int_{\cO} \partial_{i,y}
G^n (t,x,s,y)w''^{,n}_{i,s} (y)dy ds
\\&\ \ +\sum_{i=1}^{+\infty}\int_0^t\int_{\cO} G^n (t,x,s,y)w^n_{i,s} (y) dy dB^i_s .\\
\end{split}
\end{equation}
Thanks to Lemma \ref{convGreen} and the Gaussian estimates \eqref{Gaussestimate}, we deduce by the dominated convergence Theorem  that the first, second and fourth terms in the right member of \eqref{approx2} converge to the corresponding term in the expression \eqref{expressionofu} of $u$.  In  order to study the third one, we put for all $n$:
\[ z^n =-\sum_{i=1}^{d}\int_0^t\int_{\cO} \partial_{i,y}
G^n (t,x,s,y)w''^{,n}_{i,s} (y)dy ds=\sum_{i=1}^{d}\int_0^t\int_{\cO} G^n (t,x,s,y) \partial_{i,y}
w''^{,n}_{i,s} (y)dy ds ,\]
and
\[ z= \sum_{i=1}^{d}\int_0^t\int_{\cO} G (t,x,s,y) \partial_{i,y}
w''_{i,s} (y)dy ds.\]
By the same proof as above, we can prove that, at least for a subsequence, $(z^n )$ converges weakly in $L^2 (\Omega\times[0,T];H^1_0 (\cO))$ to an element $\tz$. But, it is easy by passing to the limit, to verify that $\tz$ is a weak solution of the equation:
\[ d\tz_t =\sum_{i,j}\partial_i a_{i,j}\partial_j \tz_t +\sum_i \partial_i w''_i ,\ \tz_0 =0.
\]
Since the weak solution is unique,  $\tz =z$. 
This permits to conclude that $\tu =u$. \\
Finally, as
\[ \lim_{n\rightarrow +\infty}\| \tu^n -u\|_T =0,\]
 to see that Proposition \ref{Ito1} remains valid, one just has to apply it to $\tu^n$ and then  pass to the limit by making
$n$ tend to $+\infty$.
\end{proof}
 We now  prove the following version of Ito's formula which is crucial to get uniform estimates of the solution.
\begin{Pro}
\label{Ito2} Let $u$  be the solution defined in Theorem \ref{MainTheo} with same hypotheses and $\varphi :{\mathbb{R}^+\times \mathbb{R}
}\rightarrow {\mathbb{R}}$ be one time differentiable with continuous derivative with respect to the first variable  and two times differentiable with continuous derivatives w.r.t. the second variable. We denote by  $\varphi ^{\prime }$ and $\varphi ^{\prime \prime }$ the derivatives of $\varphi$
with respect to the second variable and by $\frac{\partial \varphi}{\partial t}$ the partial derivative with respect to time. We assume that these derivatives are bounded   and $\varphi ^{\prime }\left(t, 0\right) =0$ for all $t\geq 0$. Then the
following relation holds a.s. for all $t\geq 0$:%
\begin{equation}
\label{Ito-varphi}
\begin{split}
\int_{{\cal O}}\varphi &\left(t, u_t\left( x\right) \right)  dx+  \int_0^t\int_{\cO}
\sum_{i,j=1}^d a_{i,j} (s,x)\varphi ^{\prime\prime }\left( s,u_s (x)\right )\partial_i u_s (x) \partial_j u_s (x) dx ds  =\int_{{\cal O}
}\varphi \left(0, \xi
\left( x\right) \right) dx\\&+\int_0^t \int_{\cO} \frac{\partial \varphi}{\partial s} (s,u_s (x))dxds+
\int_0^t \int_{\cO}
\varphi ^{\prime } \left(s,u_s (x) \right) \, w_s (x)  \, dx \, ds \\
& -\sum_{i=1}^d\int_0^t \int_{\cO} \partial _i  \varphi ^{\prime
}\left( s,u_s (x) \right) \, w_{i,s}^{\prime\prime} (x) dx ds +\sum_i
\int_0^t \int_{\cO} \varphi ^{\prime }(s,u_s(x))
w_{i,s}  (x) \, dx dB^i_s \\
& +\frac12\sum_{i=1}^{+\infty} \int_0^t \int_{\cO} \varphi ^{\prime \prime } (s,u_s(x)) \left( w_{i,s} (x)\right)^2 \, dx\, ds .\\
\end{split}
\end{equation}
\end{Pro}
\begin{proof} First of all, let us mention that due the boundedness  of the derivatives of $\varphi$ and the integrability conditions on $w,w'$ and $w''$, each of the terms in \eqref{Ito-varphi} are well defined. We consider the same approximation $(\tu_t^n)_{t\geq 0}$ as in
the proof of Theorem \ref{MainTheo} and we keep the same notations. 
We know that $\tu^n$ is a $H_0^1$-valued semimartingale and that it admits the decomposition given by \eqref{SemMartin2}.
  We apply the classical It\^{o}'s formula and then integrate w.r.t.
  $x$,
 this yields:
\begin{equation*}
\begin{split}
\int_{\cO}&\varphi \left( t,\tu^n_t\left( x\right) \right)dx + \sum_{i=1}^d \int_0^t\int_{\cO}
\tA_i^n (s,x)\partial_i \varphi ^{\prime }\left( s,\tu^n_s (x)\right) \, dxds
=\int_{{\cal O} }\varphi \left(0, \xi^n \left( x\right) \right)
dx\\& +\int_0^t \int_{\cO} \frac{\partial \varphi}{\partial s} (s,\tu^n_s (x))dxds+\int_0^t \int_{\cO}
\varphi ^{\prime } \left(s,\tu^n_s (x) \right) \, \tw^{\prime ,n}_s (x)  \, dx \, ds\\&
-\sum_{i=1}^d\int_0^t \int_{\cO} \partial _i  \varphi ^{\prime
}\left( s,\tu^n_s (x) \right) \, \tw_{i,s}^{\prime\prime ,n} (x) dx ds
\ + \sum_{i=1}^{+\infty} \int_0^t \int_{\cO}\varphi' (s, \tu^n_s (x) )\tw^{n}_{i,s}  (x)\,dx dB^i_s \\
& \ \ +\frac12\sum_{i=1}^{+\infty} \int_0^t \int_{\cO} \varphi ^{\prime \prime }
(s, \tu^n_s(x)) \left( \tw_{i,s}^{n } (x)\right)^2 \, dx\, ds .\\
\end{split}
\end{equation*}
By extracting subsequences, we can assume that $(\tu^n)_n$ and 
$(\tA_i^n )_n$ converge in $L^2 (\Omega\times [0,T]\times \cO )$ and $dt\times dx \times P$-almost everywhere respectively to $\tu$ and $\sum_{j=1}^d a_{i,j}\partial_j u$, so that we can apply the dominated convergence Theorem in each term of the previous equality and obtain the result in the general case.
\end{proof}
\subsection{Existence and uniqueness of the solution in $\cH$ under {\bf (HD2)} and {\bf (HI2)}}
The aim of this section is to prove existence and uniqueness of the solution of $\eqref{e1}$ with zero Dirichlet condition on the boundary under usual $L^2$-integrability conditions and assumption {\bf (H)}.\\
So, all along this section, we assume that hypotheses {\bf (H)},  {\bf (HD2)} and {\bf (HI2)} hold.
\begin{Pro}
Notions of {\rm mild solution} and {\rm weak solution} coïncide.
\end{Pro}
\begin{proof}
The fact that any mild solution is a weak solution follows from Theorem \ref{MainTheo}.

Conversely, assume that $u$ is a weak solution and define the process
\begin{equation}\label{mild} \begin{split}
v_t (\cdot )=&\int_{\cO}G(t,\cdot,0,y)\xi(y)\, dy
+\int_0^t\int_{\cO} G(t,\cdot,s,y)f(s,y ,u_s (y),\nabla u_s (y))dy
ds \\
&\ \ +\sum_{i=1}^{d}\int_0^t\int_{\cO}
G(t,\cdot,s,y)\partial_{i,y}g_i(s,. ,u_s ,\nabla u_s )(y)dy ds
\\&\ \ +\sum_{i=1}^{+\infty}\int_0^t\int_{\cO} G(t,\cdot,s,y)h_i(s,y ,u_s (y),\nabla
u_s (y))dB^i_s .\\
\end{split}
\end{equation}
We should prove that $u=v.$ Comparing the value of the integral
\[
\int_{0}^{\infty }\int_\cO [u_{s}(x)\partial _{s}\varphi (x) -\sum_{i,j=1}^d %
a_{i,j}(s,x)\partial_i u_{s} (x)\partial_j \varphi _{s}(x) ]dxds,
\]
obtained from the relation defining a weak solution, and the value of the
same integral with $v$ in the place of $u,$ given by the relation  \eqref{weak0}, we observe that the two are almost surely equal.
So, we deduce that
\[
\int_{0}^{\infty }\int_\cO (u_{s}(x)-v_s (x))\partial _{s}\varphi (x) -\sum_{i,j=1}^d %
a_{i,j}(s,x)\partial_i (u_{s} (x)-v_s (x))\partial_j \varphi _{s}(x) \, dxds=0,
\]
almost surely, for each $\varphi \in \mathcal{D}.$ Since $\mathcal{D}$
contains a countable set which is dense in it, we deduce that the relation
holds with arbitrary $\varphi \in \mathcal{D},$ outside of a negligeable set
in $\Omega .$  From this, it is standard to conclude that $u=v$ almost surely.
\end{proof}
The proof of the following Theorem is given in the Appendix \ref{Appendix2}.
\begin{Th}\label{ExistenceL2} Under hypotheses {\bf (H)}, {\bf (HD2) }and {\bf (HI2)}, equation \eqref{e1} with zero Dirichlet condition on the boundary admits a unique solution, $u$,  which belongs to $\cH$. Moreover $u$  admits $L^2 (\cO )$-continuous
trajectories and satisfies the following estimate:
\begin{equation}
\label{estimateL2}
E[\left\| u\right\| _{T}^{2}]\leq cE\left[ \left\| \xi\right\|
^{2}+ \left\| | f^0|\right\|^{2}_{2,2;T}+\left\|
|g^0 |\right\|^{2}_{2,2;T}+\left\| |h^0|\right\|^{2}_{2,2;T} \right] ,
\end{equation}
where $c$ is a constant which only depends on the structure constants.
\end{Th}
\subsection{$L^p$-estimate of the uniform norm of the solution}\label{supnorm}

As in \cite{DMS1,DMS2}, for $\theta\in [0,1)$ and $p\geq 2$ fixed, we consider the following assumptions:\\

{\bf Assumption (HI$\infty p$)}%
$$ E\left\| \xi \right\| _\infty ^p<\infty . $$
{\bf Assumption (HD$\theta p$)}
$$ E\left( \left( \left\| f^0\right\| _{\theta ;t}^{*}\right)
^p+\left( \left\| \left| g^0\right| ^2\right\| _{\theta
;t}^{*}\right) ^{\frac p2}+\left( \left\| \left| h^0\right|
^2\right\| _{\theta ;t}^{*}\right) ^{\frac p2}\right) <\infty , $$
for each $t\ge 0$. \\
Here, $\|\ \|^*_{\theta;t}$ is the functional norm similar to $\|\ \|^*_{\#;T}$ (see Appendix \ref{lpq}).\\

In \cite{DMS1}, in the case  of a SPDE driven by a finite dimensional Brownian motion and an
homogeneous second order symmetric differential operator, we have established an $L^p$-estimate of the uniform norm of the solution.  The proof of this $L^p$-estimate is based on  It\^{o}'s formula applied to the power function and the domination of the quadratic variation of the martingale part in this formula. However, the method and the technics involved to get this estimate do not depend on the dimension of the Brownian motion neither on the fact that the  matrix $a$  is homogeneous in time. Therefore, to generalize these results to our context, we  can follow
the same arguments as in \cite{DMS1} starting from Lemma 12 of this reference and this yields:

\begin{Th}\label{Lpestimate}
Assume {\bf (H)}, (\textbf{HD$\theta p$)}, \textbf{(HI$\infty p$)}
for some $\theta \in [0,1[$, $p\ge 2,$ and that the constants of the
Lipschitz conditions satisfy $\alpha +\frac{\beta ^2}2+72\beta
^2<\lambda $. Let $u= {\cal U}\left( \xi ,f,g,h\right) $, then
$$\forall t\geq 0 ,\  E\left\| u\right\| _{\infty ,\infty ;t}^p\le k\left(
t\right) E\left( \left\| \xi \right\| _\infty ^p+\left( \left\|
f^{0}\right\| _{\theta ;t}^{*}\right) ^p+\left( \left\| \left|
g^0\right| ^2\right\| _{\theta ;t}^{*}\right) ^{\frac p2}+\left(
\left\| \left| h^0\right| ^2\right\| _{\theta ;t}^{*}\right) ^{\frac
p2}\right) , $$ where $k\left( t\right) $ is a constant which depends
on the structure constants and $t.$
\end{Th}

\section{Maximum principle for local solutions}{\label{LocalSolutions}}
 In \cite{DMS2}, we have proven a maximum principle for SPDE's driven by a finite dimensional Brownian motion and
homogeneous second order symmetric differential operator. To extend these results to our context, we  follow
the same plan as in \cite{DMS2}. We mention the different  estimates who lead to the result and give the details of the proofs only  when needed.
\subsection{Estimates of the solution with null Dirichlet condition under {\bf (HD\#)}}\label{estimatediese}
The first step consists in establishing an estimate for the positive part of the solution with null Dirichlet condition. To get this estimate,
we can adapt to our case the arguments of proofs of Theorem 3, Corollary 1 and Theorem 4 in \cite{DMS2} which are based only on  estimate \eqref{estimateL2} and  It\^{o}'s formula for the solution which  do not depend on the dimension of the noise neither on the fact that the matrix $a$ is homogeneous. This yields:
\begin{Th}{\label{existence}}
Under the conditions {\bf (H)}, {\bf (HD\#) } and {\bf{(HI2)}} there
exists a unique solution $u$ of (\ref {e1}) in $\HH $.
This solution  has a version with $L^2 (\cO)$-continuous
trajectories and it satisfies the
following estimates for each $t\ge 0$ :\begin{enumerate}
\item $ E\left( \left\| u\right\| _{2,\infty ;t}^2+\left\| \nabla
u\right\| _{2,2;t}^2\right) \le k\left( t\right) E\left( \left\| \xi
\right\|_2^2+\left( \left\| f^0\right\| ^*_{\#;t}\right) ^2+\left\|
g^0\right\| _{2,2;t}^2+\left\| |h^0 |\right\| _{2,2;t}^2\right) .$\\
\item Let  $\varphi : \mathbb{R
}\rightarrow \mathbb{{R}}$ be a function of class ${\cal C}^2$ and
assume that $\varphi ^{\prime \prime }$ is bounded and $\varphi
^{\prime }\left( 0\right) =0.$ Then the
following relation holds a.s. for all $t\geq 0$:%
$$
\int_{{\cal O}}\varphi \left( u_t\left( x\right) \right) dx+\int_0^t{\cal E}%
\left( \varphi ^{\prime }\left( u_s\right) ,u_s\right) ds=\int_{{\cal O}%
}\varphi \left( \xi \left( x\right) \right) dx+\int_0^t\left(
\varphi ^{\prime }\left( u_s\right) ,f_s ( u_s ,\nabla u_s \right)
ds
$$
$$ -\int_0^t\sum_{i=1}^d\left( \partial _i\left( \varphi ^{\prime
}\left( u_s\right) \right) ,g_{i,s}(u_s ,\nabla u_s \right) ds+\frac
12\int_0^t\left( \varphi ^{\prime \prime }\left( u_s\right) ,\left|
h_{s} (u_s ,\nabla u_s )\right| ^2\right) ds$$
$$+\sum_{j=1}^{+\infty}\int_0^t\left( \varphi ^{\prime }\left( u_s\right)
,h_{j,s} (u_s ,\nabla u_s )\right) dB_s^j. $$
\item The positive part of the solution satisfies the following estimate%
$$ E\left( \left\| u^{+}\right\| _{2,\infty ;t}^2+\left\| \nabla
u^{+}\right\| _{2,2;t}^2\right) \le k\left( t\right) E\left( \left\| \xi
^{+}\right\|_2^2+\left( \left\| f^{u,0+}\right\| _{\#;t}^*\right) ^2+\left\|
g^{u,0}\right\| _{2,2;t}^2+\left\| |h^{u,0}|\right\| _{2,2;t}^2\right) , $$
\end{enumerate}
 where $k\left( t\right) $ is a constant that only
depends on $t$ and the structure constants  and
\begin{equation}
\label{f0+} \begin{split} &f^{u,0}=1_{\left\{ u>0\right\}
}f^0,\;g^{u,0}=1_{\left\{ u>0\right\} }g^0,\;h^{u,0}=1_{\left\{
u>0\right\} }h^0,\\
&f^u=f-f^0+f^{u,0},\;g^u=g-g^0+g^{u,0},\;h^u=h-h^0+h^{u,0} \\
& f^{u,0+}=1_{\left\{ u>0\right\} }\left( f^0\vee 0\right) ,\;\xi
^{+}=\xi \vee 0.\\
\end{split}
\end{equation}
\end{Th}
Let us mention that a similar relations to the one of point 3. have been obtained by Krylov, under stronger conditions (see \cite{Krylov2}, Lemma 2.4 and Lemma 2.5).
\subsection{Estimate of the positive part of a local solution}{\label{EstimatePositive}}
We first make the following remark concerning the regularity of the trajectories of any local solution.
\begin{Rq}
We have proved  in Theorem \ref{existence} that under {\bf (H)}, {\bf (HD\#) } and {\bf{(HI2)}} the
solution with null Dirichlet conditions at the boundary of $ \mathcal{O} $ has
a version with $L^2 \left(\mathcal{O }\right)$-continuous trajectories and, in
particular, that $ \lim_{t \to 0} \|u_t - \xi \|_2 =0$,
 a.s. This property extends to the local solutions in the sense
that any element of $ \mathcal{U}_{loc} (\xi,f,g,h)$ has a version
with the property that a.s. the trajectories are $L^2
\left(K\right)$-continuous, for each compact set $K \subset
\mathcal{O}$ and $$ \lim_{t \to 0} \int_K \left( u_t(x) - \xi(x) \,
\right)^2\, dx = 0.
$$
In order to see this it suffices to take a test function $ \phi \in
\mathcal{C}_c^{\infty} (\cO )$ and to verify that $ v = \phi u$ satifies the
equation
$$
dv_t = \left( Lv_t + \overline{f}_t + div \overline{g}_t \right) +
\overline{h}_t dB_t,
$$
with the initial condition $ v_0 = \phi \xi $, where
\begin{equation*}
\begin{split}
&\overline{f}_t (x) = \phi(x) f \left(t,x,u_t(x), \nabla u_t(x)
\right) - \langle \, \nabla \phi (x), \, a(x) \nabla u_t(x) \rangle
- \langle \nabla \phi (x), \, g \left(t,x, u_t(x), \nabla u_t (x)
\right)\,
\rangle ,\\
& \overline{g}_t (x) = \phi(x) g \left(t,x,u_t(x), \nabla u_t(x)
\right) - u_t (x) a(x) \nabla \phi (x) \quad \mbox{and}
\\
& \overline{h}_t (x) = \phi(x) h \left(t,x,u_t(x), \nabla u_t(x)
\right). \\
\end{split}
\end{equation*}
Thus
$v={\mathcal{U}}\left(\phi\xi,\overline{f},\overline{g},\overline{h}\right) $
 and the results of Theorem \ref{existence} hold for $v$. \\
 
\end{Rq}

Now, we consider  $u\in \mathcal{U}_{loc} (\xi ,f,g,h)$ and in order to simplify the notation we put
$$f_s = f(s,x,u_s (x),\nabla u_s (x)),\ g_s = g(s,x,u_s (x),\nabla u_s (x)), \ h_s = h(s,x,u_s (x),\nabla u_s (x)).$$
So that, $u$ is the solution of
\begin{eqnarray}{\label{lineareq}}
du_t &=& \left( \sum_{i,j=1}^d\partial_i (a_{i,j}(t,\cdot )\partial_j u_t) +f_t  +\sum_{j=1}^d \partial_j g_{j,t}\right) dt +\sum_{i=1}^{+\infty}h_{i,t} dB^i_t ,
\end{eqnarray}
with initial conditiuon $u_0=\xi$.\\
Let us remark that this technic has already been used by Krylov (\cite{Krylov2}, Lemma 2.5) in order to get Itô's formula for the non-negative part 
 of the solution. We also obtain such Itô's formula in our setting:
\begin{Le}\label{lemme14} Assume that $\partial O$ is Lipschitz,  conditions {\bf (H)}, {\bf (HD\#) } and {\bf{(HI2)}} hold.  Let $u\in \mathcal{U}_{loc} (\xi ,f,g,h)$ such  that $u^+ \in \mathcal{H}$, i.e. following Definition \ref{Def3}, $u$ is non-positive on the boundary of $\cO$.\\
Let $\varphi:\R \rightarrow \R$ be a function of class $C^2$ with bounded second order derivative and assume that $\varphi (0)=\varphi' (0)=0$. Then with the notations introduced above:
\begin{equation}{}
\begin{split}
\int_{{\cal O}}&\vphi \left( u^+_t\left( x\right) \right) dx+  \sum_{i,j=1}^d\int_0^t\int_{\cO}
\vphi'' \left( u^+_s\left( x\right) \right)a_{i,j} (s,x)\partial_i u^+_s (x) \partial_j u^+_s (x) dx ds  =\int_{{\cal O}
}\vphi\left(\xi^+
\left( x\right) \right)dx\\&+
\int_0^t \int_{\cO}
\vphi' (u^+_s (x) ) \, {f}_s (x)  \, dx \, ds
 -\sum_{i=1}^d\int_0^t \int_{\cO} \vphi'' \left( u^+_s\left( x\right) \right)\partial _i u^+_s (x) {g}_{i,s}(x) dx ds \\ &+\sum_i
\int_0^t \int_{\cO} \vphi' (u^+_s (x))
{h}_{i,s}(x) \, dx dB^i_s +\frac12\sum_{i=1}^{+\infty} \int_0^t \int_{\cO} \vphi'' \left( u^+_s\left( x\right) \right)\1_{\{u_s>0\}}\left| {h}_{i,s} (x)\right|^2 \, dx\, ds .\\
\end{split}
\end{equation}

\end{Le}
\begin{proof}
For the moment, we consider $\phi\in C^{\infty}_c (\cO )$,  $0\leq \phi\leq 1$ and put
$$\forall t\in [0,T],\ v_t =\phi u_t .$$
By a direct calculation,
we see that the process $v$
 satisfies the following equation with $\phi \xi $ as initial
data and zero Dirichlet boundary conditions,%
$$
dv _t=\left(  \sum_{i,j=1}^d\partial_i (a_{i,j}(t,\cdot )\partial_j v_t)+\widetilde{f_t}+\sum_{i=1}^d\partial _i\widetilde{%
g_{i,t}}\right) dt+\sum_{j=1}^{+\infty}\widetilde{h_{j,t}}dB_t^j, $$
where%
$$ \widetilde{f_t}=\phi f_t
-\sum_{i,j=1}^da^{i,j}(t,\cdot)\left( \partial _i\phi \right) \left(
\partial _ju_t\right) -\sum_{i=1}^d\left( \partial _i\phi \right)
g_{i,t} , $$ $$
\widetilde{g_{i,t}}=\phi g_{i,t}
-u_t\sum_{j=1}^da^{i,j}(t,\cdot )\partial _j\phi ,i=1,...d,\;\;\widetilde{h_{j,t}}%
=\phi h_{j,t} ,j\in\Ne. $$
Let us note that 
$$ E\big[ ( \left\| \tilde{f}\right\|_{\#;t}^* ) ^2+\left\|
\tilde{g}\right\| _{2,2;t}^2+\left\| \tilde{h}\right\| _{2,2;t}^2\big] <\infty
, $$ 
so we'll be able to apply  Itô's formula (point 2. of the previous Theorem).\\
Now, we approximate the function $\psi:
y\in\R\rightarrow \vphi (y^+ )$ by a sequence $(\psi_n )_{n\in\Ne}$ of
smooth functions constructed as follows:\\
 So, let $\zeta$ be a $C^{\infty}$ increasing function such that
\[\forall y\in ]-\infty ,1],\ \zeta (y)=0 \makebox{ and }
\forall y\in [2,+\infty[,\ \zeta (y)=1.\]
 We set for all
$n\in\Ne$:
\[\forall y\in\R ,\ \psi_n (y)=\vphi (y)\zeta (ny).\]
It is easy to verify that $(\psi_n )_{n\in\Ne}$ converges
uniformly to the function $\psi$, $(\psi'_n )_n$ converges everywhere to the function $(y\mapsto \vphi' (y^+ ))$ and $(\psi''_n )_n$ converges everywhere to the function $(y\mapsto \1_{\{y>0\}}\vphi'' (y^+ ))$.  Moreover we have the
estimates:
\begin{equation}\label{appxs}\forall y\in\R_+ ,\forall n\in\Ne ,\ \ 0\leq \psi_n (y)\leq\psi
(y),\ 0\leq\psi_n^{\prime}(y)\leq Cy,\ \
\|\psi_n^{\prime\prime}(y)\|\leq C,\end{equation}
where $C$ is a constant.\\
Thanks to the previous Theorem, we have for all $n$ and all $t\geq 0 :$
\begin{equation}\label{Itopsin}
\begin{split}
\int_{{\cal O}}\psi_n &\left(v_t\left( x\right) \right)  dx+  \sum_{i,j=1}^d\int_0^t\int_{\cO}
a_{i,j} (s,x)\psi_n^{\prime\prime }\left( v_s (x)\right )\partial_i v_s (x) \partial_j v_s (x) dx ds  =\int_{{\cal O}
}\psi_n\left(\phi (x)\xi
\left( x\right) \right) dx\\&+
\int_0^t \int_{\cO}
\psi_n ^{\prime } \left(v_s (x) \right) \, \widetilde{f}_s (x)  \, dx \, ds
 -\sum_{i=1}^d\int_0^t \int_{\cO}  \psi_n^{\prime\prime
}\left( v_s (x) \right) \partial _i v_s (x) \widetilde{g}_{i,s}(x) dx ds \\ &+\sum_i
\int_0^t \int_{\cO} \psi_n^{\prime }(v_s(x))
\widetilde{h}_{i,s}(x) \, dx dB^i_s +\frac12\sum_{i=1}^{+\infty} \int_0^t \int_{\cO} \psi_n^{\prime \prime } (v_s(x)) \left| \widetilde{h}_{i,s} (x)\right|^2 \, dx\, ds .\\
\end{split}
\end{equation}
Let us remark that $v\in\mathcal{H}$ so that thanks to estimates \eqref{appxs}, each term in the previous equality is well defined and even dominated in $L^1$. Let us focus on the particular term $\int_0^t \int_{\cO}
\psi_n ^{\prime } \left(v_s (x) \right) \, \widetilde{f}_s (x)  \, dx \, ds
$. We have for all $n\in\Ne$:
$$|\psi_n ^{\prime } \left(v_s  \right) \, \widetilde{f}_s| \leq C |v_s | |\widetilde{f}_s|,$$
and as $v\in \mathcal{H}\bigcap L_{\#;t}$ and $\widetilde{f}\in L^*_{\#;t}$ the Hölder inequality \eqref{Holder2} ensures that $|v \widetilde{f}|$ belongs to $L^1 (\Omega \times [0,T]\times \cO )$. The other terms being easier to dominate,  by the dominated convergence Theorem
 and using the fact that
$\1_{\{v_s>0\}}\partial_i v_s =\partial_i v^+_s ,$ we get as $n$ tends to $+\infty$:
\begin{equation}{\label{Iton}}
\begin{split}
\int_{{\cal O}}&\vphi \left( v^+_t\left( x\right) \right) dx+ \sum_{i,j=1}^d \int_0^t\int_{\cO}
\vphi'' \left( v^+_s\left( x\right) \right)a_{i,j} (s,x)\partial_i v^+_s (x) \partial_j v^+_s (x) dx ds  =\int_{{\cal O}
}\vphi\left(\phi (x)\xi^+
\left( x\right) \right)dx\\&+
\int_0^t \int_{\cO}
\vphi' (v^+_s (x) ) \, \widetilde{f}_s (x)  \, dx \, ds
 -\sum_{i=1}^d\int_0^t \int_{\cO} \vphi'' \left( v^+_s\left( x\right) \right)\partial _i v^+_s (x) \widetilde{g}_{i,s}(x) dx ds \\ &+\sum_i
\int_0^t \int_{\cO} \vphi' (v^+_s (x))
\widetilde{h}_{i,s}(x) \, dx dB^i_s +\frac12\sum_{i=1}^{+\infty} \int_0^t \int_{\cO} \vphi'' \left( v^+_s\left( x\right) \right)\1_{\{v_s>0\}}\left| \widetilde{h}_{i,s} (x)\right|^2 \, dx\, ds .\\
\end{split}
\end{equation}
Consider now a sequence $(\phi_n )_n$ of non-negative functions in $C_0^{\infty} (\cO)$, $0\leq \phi_n\leq 1$  $\forall n\in\Ne$ converging to $1$ everywhere on $\cO$ and such that for any $w\in H_0^1 (\cO )$ the sequence
$(\phi_n  w)_n$ tends to $w$ in $H_0^1 (\cO )$ and
$$\sup_n \| \phi_n w\|_{H^1_0 (\cO)}\leq C \| w\|_{H^1_0 (\cO)},$$
where $C$ is a constant which does not depend on $w$.\\The existence of such a sequence is proved in the Appendix, Lemma \ref{Troncation}.\\
Let us remark that if $i\in \{1,\cdots ,d\}$ and $w\in H^1_0 (\cO)$, then $(w\partial_i \phi_n)_n$ tends to $0$ in $L^2 (\cO)$.\\
We set $v_n =\phi_n u$ and
$$ \widetilde{f^n_t}=\phi_n f_t
-\sum_{i,j=1}^d a^{i,j}(t,\cdot)\left( \partial _i\phi_n \right) \left(
\partial _ju_t\right) -\sum_{i=1}^d\left( \partial _i\phi_n \right)
g_{i,t} , $$ $$
\widetilde{g^n_{i,t}}=\phi_n g_{i,t}
-u_t\sum_{j=1}^d a^{i,j}(t,\cdot )\partial _j\phi_n ,i=1,...d,\;\;\widetilde{h^n_{j,t}}%
=\phi_n h_{j,t} ,j\in\Ne. $$
We now apply relation \eqref{Iton} to $v_n$ and get
\begin{equation}{\label{Itonn}}
\begin{split}
\int_{{\cal O}}&\vphi \left( v^+_{n,t}\left( x\right) \right) dx+ \sum_{i,j=1}^d \int_0^t\int_{\cO}
\vphi'' \left( v^+_{n,s}\left( x\right) \right)a_{i,j} (s,x)\partial_i v^+_{n,s} (x) \partial_j v^+_{n,s} (x) dx ds  =\int_{{\cal O}
}\vphi\left(\phi_n (x)\xi^+
\left( x\right) \right)dx\\&+
\int_0^t \int_{\cO}
\vphi' (v^+_{n,s} (x) ) \, \widetilde{f}_s^n (x)  \, dx \, ds
 -\sum_{i=1}^d\int_0^t \int_{\cO} \vphi'' \left( v^+_{n,s}\left( x\right) \right)\partial _i v^+_s (x) \widetilde{g}_{i,s}^n (x) dx ds \\ &+\sum_i
\int_0^t \int_{\cO} \vphi' (v^+_{n,s} (x))
\widetilde{h}_{i,s}^n (x) \, dx dB^i_s +\frac12\sum_{i=1}^{+\infty} \int_0^t \int_{\cO} \vphi'' \left( v^+_{n,s}\left( x\right) \right)\1_{\{v_{n,s}>0\}}\left| \widetilde{h}_{i,s}^n (x)\right|^2 \, dx\, ds .\\
\end{split}
\end{equation}
We have
\begin{eqnarray*}
\varphi' (v^+_{n,s})   \widetilde{f}_s^n -\sum_{i=1}^d  \varphi'' (v^+_{n,s})\partial _i v^+_{n,s}  \widetilde{g^n}_{i,s}&=& \varphi' (v^+_{n,s}) \phi_n f_s-\sum_{i,j} a_{i,j}(s)\varphi' (v^+_{n,s} )\partial_j\phi_n\partial_i u^+_s \\&& \hspace{-5cm}+\sum_{i,j}a_{i,j}(s)\varphi'' (v^+_{n,s})u^+_s \partial_i v^+_{n,s} \partial_j \phi_n  -\sum_i (\varphi' (v^+_{n,s}) g_{i,s}\partial_i \phi +  \varphi'' (v^+_{n,s})\phi_n g_{i,s}\partial_i v^+_{n,s}).
\end{eqnarray*}
By remarking for example that for all $s\in (0,T]$ $(\phi_n \varphi' (v^+_{n,s})_n$ (resp. $(\partial_i \phi_n \varphi'(v^+_{n,s}))_n$) tends to $(\varphi' (u^+_s ))$ (resp. $0$)
in $H^1_0 (\cO)$ (resp. in $L^2 (\cO )$) we conclude, thanks to the dominated convergence Theorem, by  making $n$ tend to $+\infty$ in \eqref{Itonn}.
\end{proof}

\subsection{A comparison Theorem}
By applying the previous It\^{o}'s formula for  $\varphi (x)=x^2$, as in Theorem \ref{existence}, the above Proposition leads to the following generalization of the
estimate of the positive part:
\begin{Cor}
\label{positive-estimate}
Under the hypotheses of Lemma \ref{lemme14} with same notations, one has the following estimate: %
$$ E\left( \left\| u^{+}\right\| _{2,\infty ;t}^2+\left\| \nabla
u^{+}\right\| _{2,2;t}^2\right) \le k\left( t\right) E\left( \left\|
\xi ^{+}\right\| _2^2+\left( \left\| f^{u,0+}\right\|
^*_{\#;t}\right) ^2+\left\| g^{u,0}\right\| _{2,2;t}^2+\left\|
h^{u,0}\right\| _{2,2;t}^2\right) . $$
\end{Cor}
The key point of the proof of the maximum principle is the following comparison Theorem which is an immediate consequence of the previous estimate:

\begin{Th}{\label{comparison}} Assume that $\partial O $ is Lipschitz.
 Let  $f^1,$ $f^2$ be two functions similar to $f$ which
satisfy the
Lipschitz condition \textbf{(H)-(i)}, $g$ (resp. $h$) satisfies \textbf{(H)-(ii)} (resp. \textbf{(H)-(iii)} and assume that both triples $\left( f^1,g,h\right) $ and $%
\left( f^2,g,h\right) $ satisfy \textbf{(HD\#)}. Let $\xi
^1,\xi ^2$ two random variables similar to $\xi $
satisfying \textbf{(HI)}. Let $u^i\in {\cal U}_{loc}\left( \xi
^i,f^i,g,h\right) ,i=1,2$ and suppose that the process $\left(
u^1-u^2\right) ^{+}$ belongs to $\HH
$ and that one has%
$$ E\left( \left\| f^1\left(.,., u^2,\nabla u^2\right) -f^2\left(.,.,
u^2,\nabla u^2\right) \right\| ^*_{\#;t}\right) ^2<\infty ,\;\;
\mbox{ for all} \quad t\ge 0.$$ If $\xi ^1\le \xi ^2$ a.s. and
$f^1\left(t,\omega, u^2,\nabla u^2\right) \le f^2\left(t,\omega,
u^2,\nabla u^2\right) $,  $dt\otimes dx \otimes dP$-a.e., then one
has $$u^1 (t,x)\le u^2 (t,x) \quad\ \ dt \otimes dx\otimes P\makebox{-a.e.}$$
\end{Th}

\subsection{The maximum principle}
As in Subsection \ref{supnorm}, we work under assumptions (\textbf{HD$\theta p$)} and \textbf{(HI$\infty p$)}. The following property has been proved in \cite{DMS2}, Lemma 2:

$$\left\| u\right\| _{1,1;T}\le c\left\| u\right\| _{\theta ;T}^{*},$$
for some constant $c>0$. As a consequence, {\bf(HD$\theta p$)}is stronger than {\bf(HD\#)}.

We first consider the case of a solution $u$ such that $u\leq 0$ on $\partial \cO$.

\begin{Th} Assume that $\partial O$ is Lipschitz, that
 {\bf (H)}, (\textbf{HD$\theta p$)}, \textbf{(HI$\infty p$)} hold
for some $\theta \in [0,1[$, $p\ge 2,$ and that the constants of the
Lipschitz conditions satisfy $\alpha +\frac{\beta ^2}2+72\beta
^2<\lambda $. Let $u\in {\cal U}_{loc}\left( \xi ,f,g,h\right) $ be
such that
$u^{+}\in \HH.$  Then one has%
$$ E\left\| u^{+}\right\| _{\infty ,\infty ;t}^p\le k\left(
t\right) E\left( \left\| \xi ^{+}\right\| _\infty ^p+\left( \left\|
f^{0,+}\right\| _{\theta ;t}^{*}\right) ^p+\left( \left\| \left|
g^0\right| ^2\right\| _{\theta ;t}^{*}\right) ^{\frac p2}+\left(
\left\| \left| h^0\right| ^2\right\| _{\theta ;t}^{*}\right) ^{\frac
p2}\right) , $$ where $k\left( t\right) $ is constant that depends
of the structure constants and $t\ge 0.$
\end{Th}

\begin{proof}

Set $v={\cal U}\left( \xi ^{+},\widehat{f},g,h\right) $ the
solution with zero Dirichlet boundary conditions, where the
function $\widehat{f}$ is defined by $\widehat{f}=f+f^{0,-},$ with
$f^{0,-}=0\vee \left( -f^0\right) .$ The assumption on the
Lipschitz constants ensure the applicability of Theorem \ref{Lpestimate}, which gives the estimate%
$$ E\left\| v\right\| _{\infty ,\infty ;t}^p\le k\left( t\right)
E\left( \left\| \xi ^{+}\right\| _\infty ^p+\left( \left\|
f^{0,+}\right\| _{\theta ;t}^{*}\right) ^p+\left( \left\| \left|
g^0\right| ^2\right\| _{\theta ;t}^{*}\right) ^{\frac p2}+\left(
\left\| \left| h^0\right| ^2\right\| _{\theta ;t}^{*}\right) ^{\frac
p2}\right) , $$ because $\widehat{f}^0=f^{0,+}.$ Then $\left(
u-v\right) ^{+}\in \HH $ and we observe that all the conditions of
the
preceding theorem are satisfied so that we may apply it and deduce that $%
u\le v.$ This implies $u^{+}\le v^{+}$ and the above estimate of
$v$ leads to the asserted estimate.
\end{proof}

Let us generalize the previous result by considering a real It\^{o} process of the form%
$$ M_t=m+\int_0^tb_sds+\sum_{j=1}^{+\infty}\int_0^t\sigma
_{j,s}dB_s^j, $$
where $m$ is a real random variable and $b=\left( b_t\right) _{t\ge 0},$ $%
\sigma =\left( \sigma _{1,t},\cdots ,\sigma _{n,t}\cdots\right) _{t\ge 0}$
are adapted processes.

\begin{Th}
\label{maintheo}
 Assume {\bf (H)}, (\textbf{HD$\theta p$)}, \textbf{(HI$\infty p$)}
for some $\theta \in [0,1[$, $p\ge 2,$ and that the constants of the
Lipschitz conditions satisfy $\alpha +\frac{\beta ^2}2+72\beta
^2<\lambda $. Assume also that $m$ and the processes $b$ and $\sigma
$
satisfy the following integrability conditions%
$$ E\left| m\right| ^p<\infty ,\;E\left( \int_0^t\left| b_s\right|
^{\frac 1{1-\theta }}ds\right) ^{p\left( 1-\theta \right) }<\infty
,\;E\left(
\int_0^t\left| \sigma _s\right| ^{\frac 2{1-\theta }}ds\right) ^{\frac{%
p\left( 1-\theta \right) }2}<\infty , $$ for each $t\ge 0.$ Let
$u\in {\cal U}_{loc}\left( \xi ,f,g,h\right) $ be such that $\left(
u-M\right) ^{+}$ belongs to $\HH .$
Then one has%
\begin{equation*}
\begin{split}
E\left\| \left( u-M\right) ^{+}\right\| _{\infty ,\infty ;t}^p & \le k\left(
t\right) E\Big[ \left\| \left( \xi -m\right) ^{+}\right\| _\infty ^p +\left(
\left\| \Big(f (\cdot,\cdot,M,0) - b \Big)^+ \, \right\|
_{\theta;t}^{*}\right) ^p \, + \, \\
& \quad \quad \quad  \left( \left\| \big| g(\cdot,\cdot,M,0)\big|^2\right\|
_{\theta ;T}^{*}\right)^{\frac p2}
 +\left( \left\| \left| h (\cdot,\cdot,M,0) - \sigma \right| ^2\right\|_{\theta ;T}^{*}\right)^{\frac{p}{2}} \, \Big]\\
\end{split}
\end{equation*}
 where $k\left( t\right) $ is the constant from the
preceding corollary. \end{Th}

\begin{Rq}The right hand side of this estimate is
dominated by the following quantity which is expressed directly in
terms of the characteristics of the process $M$,
\begin{equation*}
\begin{split}
&k\left( t\right) E \,\Big[ \left\| \left( \xi -m\right)
^{+}\right\| _\infty ^p+\left| m\right| ^p+\left( \left\|
f^{0,+}\right\| _{\theta ;t}^{*}\right) ^p+\left( \left\| \left|
g^0\right| ^2\right\| _{\theta ;T}^{*}\right) ^{\frac p2} +\left(
\left\| \left| h^0\right| ^2\right\| _{\theta ;T}^{*}\right) ^{\frac
p2}
\\
& +\left( \int_0^t\left| b_s\right| ^{\frac 1{1-\theta }}ds\right)
^{p\left( 1-\theta \right) }+\left( \int_0^t\left| \sigma _s\right|
^{\frac 2{1-\theta }}ds\right) ^{\frac{p\left( 1-\theta \right)
}2}\, \Big].
\end{split}
\end{equation*}
\end{Rq}

\section{Appendix}{\label{Appendix}}
\subsection{Functional spaces}\label{lpq}

We just recall the main definitions, all the details may be found in \cite{DMS1} and \cite{DMS2}.\\
Let $\left( p_1,q_1\right) ,\left( p_2,q_2\right) \in \left[
1,\infty
\right] ^2$ be fixed and set%
$$ I=I\left( p_1,q_1,p_2,q_2\right) :=\left\{ \left( p,q\right)
\in \left[ 1,\infty \right] ^2/\;\exists \;\rho \in \left[
0,1\right] s.t.\right. $$ $$ \left. \frac 1p=\rho \frac
1{p_1}+\left( 1-\rho \right) \frac 1{p_2},\frac 1q=\rho \frac
1{q_1}+\left( 1-\rho \right) \frac 1{q_2}\right\} . $$ This means
that the set of inverse pairs $\left( \frac 1p,\frac 1q\right) ,$
$(p,q)$
 belonging to $I,$ is a segment contained in the square $%
\left[ 0,1\right] ^2,$ with the extremities $\left( \frac
1{p_1},\frac 1{q_1}\right) $ and $\left( \frac 1{p_2},\frac
1{q_2}\right) .$ \\
We introduce:
$$ L_{I;t}=\bigcap_{\left( p,q\right) \in I}L^{p,q}\left( \left[
0,t\right] \times {\cal O}\right) . $$ We know that this space  coincides with the intersection of the extreme spaces,%
$$ L_{I;t}=L^{p_1,q_1}\left( \left[ 0,t\right] \times {\cal
O}\right) \cap L^{p_2,q_2}\left( \left[ 0,t\right] \times {\cal
O}\right) $$
and that it is a Banach space with the following norm%
$$ \left\| u\right\| _{I;t}:=\left\| u\right\| _{p_1,q_1;t}\vee
\left\| u\right\| _{p_2,q_2;t}. $$
we also need  the algebraic sum%
$$ L^{I;t}:=\sum_{\left( p,q\right) \in I}L^{p,q}\left( \left[
0,t\right] \times {\cal O}\right) .$$ It is a normed vector space with the norm%
$$ \left\| u\right\|^{I;t} :=\,\inf \left\{ \sum_{i=1}^n\left\|
u_i\right\| _{r_i,s_i;\,t}\,/\;u=\sum_{i=1}^nu_i,u_i\in
L^{r_i,s_i}\left( \left[ 0,t\right] \times {\cal O}\right) ,\,
\left( r_i,s_i\right) \in I,\,i=1,...n;\,n\in
\mathbb{{N}}^{*}\right\} . $$
Clearly one has $L^{I;t}\subset L^{1,1}\left( \left[ 0,t\right] \times {\cal %
O}\right) $ and $\left\| u\right\| _{1,1;t}\le c\left\| u\right\|
^{I;t},$ for each $u\in L^{I;t},$ with a certain constant $c>0.$

We also remark that if $\left( p,q\right) \in I,$ then the conjugate pair $%
\left( p^{\prime },q^{\prime }\right) ,$ with $\frac 1p+\frac
1{p^{\prime }}=\frac 1q+\frac 1{q^{\prime }}=1,$ belongs to another
set, $I^{\prime },$
of the same type. This set may be described by%
$$ I^{\prime }=I^{\prime }\left( p_1,q_1,p_2,q_2\right) :=\left\{
\left( p^{\prime },q^{\prime }\right) /\;\exists \left( p,q\right)
\in I\;s.t.\;\frac 1p+\frac 1{p^{\prime }}=\frac 1q+\frac
1{q^{\prime }}=1\right\} $$ and it is not difficult to check that
$I^{\prime }\left( p_1,q_1,p_2,q_2\right) =I\left( p_1^{\prime
},q_1^{\prime },p_2^{\prime },q_2^{\prime }\right) ,$ where
$p_1^{\prime },q_1^{\prime },p_2^{\prime }$ and $q_2^{\prime }$ are
defined by $\frac 1{p_1}+\frac 1{p_1^{\prime }}=\frac 1{q_1}+\frac
1{q_1^{\prime }}=\frac 1{p_2}+\frac 1{p_2^{\prime }}=\frac
1{q_2}+\frac 1{q_2^{\prime }}=1.$

Moreover, by H\"older's inequality, it follows that one has
\begin{equation}
\label{dual}\int_0^t\int_{{\cal O}}u\left( s,x\right) v\left(
s,x\right) dxds\le \left\| u\right\| _{I;t}\left\| v\right\|
^{I^{\prime };t},
\end{equation}
for any $u\in L_{I;t}$ and $v\in L^{I^{\prime };t}.$ This inequality
shows
that the scalar product of $L^2\left( \left[ 0,t\right] \times {\cal O}%
\right) $ extends to a duality relation for the spaces $L_{I;t}$ and
$L^{I^{\prime };t}.$

Now let us recall that the Sobolev inequality states that%
$$ \left\| u\right\| _{2^{*}}\le c_S\left\| \nabla u\right\| _2,
$$ for each $u\in H_0^1\left( {\cal O}\right) ,$ where $c_S>0$ is
a constant
that depends on the dimension and $2^{*}=\frac{2d}{d-2}$ if $d>2,$ while $%
2^{*}$ may be any number in $]2,\infty [$ if $d=2$ and $2^{*}=\infty $ if $%
d=1$ (see for example \cite{Evans}, Chapter 5). Therefore one has%
$$ \left\| u\right\| _{2^{*},2;t}\le c_S\left\| \nabla u\right\|
_{2,2;t}, $$ for each $t\ge 0$ and each $u\in L_{loc}^2\left(
\mathbb{R}_{+};H_0^1\left( {\cal O}\right) \right) .$ And if $u\in
L_{loc}^{\infty}\left( \mathbb{R}_{+}; L^2\left( {\cal O}\right) \,
\right) \bigcap L^2_{loc} \left( \mathbb{R}_+;  H_0^1\left( {\cal
O}\right) \right),$ one has
$$\left\| u\right\| _{2,\infty ;t}\vee \left\| u\right\|
_{2^{*},2;t}\le c_1\left( \left\| u\right\| _{2,\infty ;t}^2+\left\|
\nabla u\right\| _{2,2;t}^2\right) ^{\frac 12},$$with $c_1=c_S\vee
1.$

One particular case of interest for us in relation with this
inequality is when $p_1=2,q_1=+\infty $ and $p_2=2^{*},q_2=2.$ If
$I=I\left( 2,\infty ,2^{*},2\right) ,$ then the corresponding set of
associated conjugate numbers is $I^{\prime }=I^{\prime }\left(
2,\infty ,2^{*},2\right) =I\left( 2,1,\frac{2^{*}}{2^{*}-1},2\right)
,$ where for $d=1$ we make the convention that
$\frac{2^{*}}{2^{*}-1}=1.$ In this particular case we shall use the
notation $L_{\#;t}:=L_{I;t}$ and $L_{\#;t}^*:=L^{I^{\prime };t}$ and
we recall that we have introduced the following norms%
$$ \left\| u\right\| _{\#;t}:=\left\| u\right\| _{I;t}=\left\|
u\right\| _{2,\infty ;t}\vee \left\| u\right\|
_{2^{*},2;t},\;\left\| u\right\|_{\#;t}^*:=\left\| u\right\|
^{I^{\prime };t}. $$ Thus we may write
\begin{equation}
\label{sobolev}\left\| u\right\| _{\#;t}\le c_1\left( \left\|
u\right\| _{2,\infty ;t}^2+\left\| \nabla u\right\|
_{2,2;t}^2\right) ^{\frac 12},
\end{equation}
for any  $u\in L_{loc}^{\infty}\left( \mathbb{{R}}_{+}; L^2\left(
{\cal O}\right) \, \right) \bigcap L^2_{loc} \left( \mathbb{R}_+;
H_0^1\left( {\cal O}\right) \right)$ and $t\ge 0$ and the
duality inequality becomes%
$$ \int_0^t\int_{{\cal O}}u\left( s,x\right) v\left( s,x\right)
dxds\le \left\| u\right\| _{\#;t}\left\| v\right\|_{\#;t}^*, $$ for
any $u\in L_{\#;t}$ and $v\in L_{\#;t}^*.$

For $d\ge 3$ and some parameter
$\theta \in [0,1[$ we
used the notation%
$$ \Gamma _\theta ^{*}=\left\{ \left( p,q\right) \in \left[
1,\infty \right] ^2/\;\frac d{2p}+\frac 1q=1-\theta \right\} , $$
$$ L_\theta ^{*}=\sum_{\left( p,q\right) \in \Gamma _\theta
^{*}}L^{p,q}\left( \left[ 0,t\right] \times {\cal O}\right) $$ $$
\left\| u\right\| _{\theta ;t}^{*}:=\,\inf \left\{
\sum_{i=1}^n\left\| u_i\right\|
_{p_i,q_i;\,t}\,/\;u=\sum_{i=1}^nu_i,u_i\in L^{p_i,q_i}\left(
\left[ 0,t\right] \times {\cal O}\right) ,\right. $$ $$
\left. \left( p_i,q_i\right) \in \Gamma _\theta ^{*},\,i=1,...n;\,n\in {\bf N%
}^{*}\right\} . $$
If $d=1,2.$ we put
 $$ \Gamma _\theta
^{*}=\left\{ \left( p,q\right) \in \left[ 1,\infty \right]
^2/\;\frac{2^{*}}{2^{*}-2}\frac 1p+\frac 1q=1-\theta \right\} $$
with the convention $%
\frac{2^{*}}{2^{*}-2}=1$ for $d=1.$

We want to express these quantities in the new notation introduced
in the subsection \ref{lpq} and to compare the norms $\left\|
u\right\|
_{\theta ;t}^{*}$ and $\left\| u\right\| ^*_{\#;t}.$ So, we first remark that $%
\Gamma _\theta ^{*}=I\left( \infty ,\frac 1{1-\theta },\frac
d{2\left( 1-\theta \right) },\infty \right) $ and that the norm
$\left\| u\right\| _{\theta ;t}^{*}$ coincides with $\left\|
u\right\| ^{\Gamma _\theta ^{*};t}=\left\| u\right\| ^{I\left(
\infty ,\frac 1{1-\theta },\frac d{2\left( 1-\theta \right)
},\infty \right) ;t}.$ On the other hand, we
recall that the norm $\left\| u\right\| ^*_{\#;t}$ is associated to the set $%
I\left( 2,1,\frac{2^{*}}{2^{*}-1},2\right) ,$ i.e. $\left\|
u\right\|
^*_{\#;t} $ coincides with $\left\| u\right\| ^{I\left( 2,1,\frac{2^{*}}{%
2^{*}-1},2\right) ;t}.$
\subsection{Proof of Proposition \ref{opU}}{\label{proofopU}}
Assume first that $w\in \left( C_c^{\infty} (\R_+)\otimes H^1_0 (\cO )\right)^d $.
In this case, the fact that $u$ is the weak solution of the given equation  and satisfies equality \eqref{Append1} i.e.
\begin{equation*} \frac 12 \|u_t \|^2 +\int_0^t \sum_{i,j=1}^d\int_{\cO}
a_{i,j}(s,x)\partial_i u_s (x)\partial_j u_s (x)\, dx\, ds
=-\sum_{i=1}^d \int_0^t \left(  \tilde{w}_{i,s} ,\partial_iu_s \right)\,
ds, \ \ t\geq 0\end{equation*} is a consequence of Theorem \ref{MainTheo} and Proposition \ref{Ito2}
with $w'=\partial \tilde{w}$ and $\xi=w''=w=0$.\\
Then, thanks to the ellipticity assumptions, we get:
\begin{eqnarray*}\frac12 \| u_t \|^2 +\lambda \int_0^t \|\nabla u_s \|^2 ds &\leq &\int_0^t |\sum_{i=1}^d (w_{i,s}, \partial_i u_s ) |\,  ds\\
&\leq& \frac{\lambda}{2} \int_0^t \|\nabla u_s \|^2 ds +\frac{8}{\lambda} \int_0^t \| |w_s |\|^2 \, ds.
\end{eqnarray*}
from this we clearly get estimate \eqref{Append2}.\\
The general case is obtained by approximating $\tilde{w}$ by a sequence of elements in $ \left( C_c^{\infty} (\R_+)\otimes H^1_0 (\cO )\right)^d$.
\subsection{Proof of Theorem \ref{ExistenceL2}}\label{Appendix2}

 We  keep the same notations as in Theorem \ref{MainTheo}.\\ Let $\gamma$ and $\delta$ be 2 positive constants. On  $F_T$, we introduce the norm
\[\forall u\in F_T ,\
\parallel u\parallel_{\gamma,\delta}=E\left( \int_0^T e^{-\gamma t}\left(\delta\parallel u_t
\parallel^2+\| \nabla u_t \|^2\right)\, dt\right) .\]
It is clear that $\parallel \cdot \parallel_{\gamma ,\delta}$ is
equivalent to $\parallel \cdot\parallel_{F_T}$. We consider the
map, $\Lambda$, from $F_T$ into $F_T$ defined by:\\ $\forall u\in F_T
,\forall (t,x)\in [0,T]\times \cO :$
\begin{equation}\label{mild} \begin{split}
\Lambda (u)(t,x )=&\int_{\cO}G(t,\cdot,0,y)\xi(y)\, dy
+\int_0^t\int_{\cO} G(t,\cdot,s,y)f(s,y ,u_s (y),\nabla u_s (y))dy
ds \\
&\ \ +\sum_{i=1}^{d}\int_0^t\int_{\cO}
G(t,\cdot,s,y)\partial_{i,y}g_i(s,. ,u_s ,\nabla u_s )(y)dy ds
\\&\ \ +\sum_{i=1}^{+\infty}\int_0^t\int_{\cO} G(t,\cdot,s,y)h_i(s,y ,u_s (y),\nabla
u_s (y))dB^i_s .\\
\end{split}
\end{equation}
Let $u$ and $v$ be in $F_T$. We put:
\[\forall s\in [0,T],\, \fb_s =f(s,\cdot ,u_s ,\nabla u_s )-f(s,\cdot ,v_s ,\nabla v_s
),\]
\[\forall s\in [0,T],\, \gb_s =g(s,\cdot ,u_s ,\nabla u_s )-g(s,\cdot ,v_s ,\nabla v_s
),\]
\[\forall s\in [0,T],\, \hb_s =h(s,\cdot ,u_s ,\nabla u_s )-h(s,\cdot ,v_s ,\nabla v_s
),\]
and $\forall t\in [0,T],$
\begin{eqnarray*} \ub_t
&=&\Lambda(u)_t -\Lambda(v)_t\\
&=& \int_0^t\int_{\cO} G(t,\cdot,s,y)\fb_s (y) \, dy\,
ds  +\sum_{i=1}^{d}\int_0^t\int_{\cO}
G(t,\cdot,s,y)\partial_{i,y}\gb_{i,s} (y)\,dy \,ds
\\&&\quad+\sum_{i=1}^{+\infty}\int_0^t\int_{\cO} G(t,\cdot,s,y)\hb_{i,s} (y)\, dy\,dB^i_s .
\end{eqnarray*}
 By  It\^{o}'s formula \eqref{Ito-varphi}, we get
\begin{equation*}
\begin{split}
e^{-\gamma T}\|\ub_T\|^2 & +2\int_0^T e^{-\gamma s}\int_{\cO}\sum_{i,j} a_{i,j}(s,x)\partial_i \ub_s (x)\partial_j \ub_s (x)\, dx ds=-\gamma \int_0^T e^{-\gamma s}\|\ub_s \|^2
ds \\
&+ 2\int_0^T e^{-\gamma s}\left(\ub_s ,\fb_s \right)   ds
-2\sum_{i=1}^d \int_0^T e^{-\gamma s}\left(\partial_i \ub_s ,\gb_{i,s} \right)  ds \\
& \quad +2\sum_{i=1}^{+\infty}\int_0^T e^{-\gamma s}\left(\ub_s ,\hb_{i,s }\right) dB_s+\int_0^T
e^{-\gamma s}\| \hb_s \|^2  ds.
\end{split}
\end{equation*}
%
Using hypotheses on
$f$, $g$, $h$ we have for all $\varepsilon>0$
\begin{eqnarray*}
2\int_0^T e^{-\gamma s}(\ub_s ,\fb_s )\,
ds&\leq&1/\varepsilon \int_0^T e^{-\gamma s} \parallel\ub_s
\parallel^2 \, ds+\varepsilon \int_0^T e^{-\gamma s} \parallel\fb_s
\parallel^2 \, ds\\
&\leq&1/\varepsilon \int_0^T e^{-\gamma s} \parallel\ub_s
\parallel^2 \, ds+C\varepsilon \int_0^T e^{-\gamma s} \parallel u_s -v_s \parallel^2
\, ds\\&&+C\varepsilon \int_0^T e^{-\gamma s}\| \nabla u_s -\nabla v_s \|^2\, ds,
\end{eqnarray*}

\begin{eqnarray*}
-2\sum_{i=1}^d \int_0^T e^{-\gamma s}\left(\partial_i \ub_s ,\gb_{i,s} \right)  \, ds&\leq&2\int_0^T e^{-\gamma s}\| \nabla \ub_s\|  (C\|u-v\| +\alpha \| \nabla u_s-\nabla v_s \| )\, ds\\&\leq& C\varepsilon \int_0^T e^{-\gamma s}\| \nabla \ub_s \|^2\, ds+\frac{C}{\varepsilon }\int_0^T e^{-\gamma s} \parallel u_s -v_s \parallel^2
\, ds\\&&+\alpha \int_0^T e^{-\gamma s}\| \nabla \ub_s \|^2\, ds+\alpha \int_0^T e^{-\gamma s}\| \nabla u_s -\nabla v_s \|^2\, ds,
\end{eqnarray*}

and

\begin{eqnarray*}
\int_0^T
e^{-\gamma s}\| \hb_s \|^2 \, ds &\leq& C(1+1/\varepsilon)\int_0^T e^{-\gamma s} \parallel u_s -v_s \parallel^2
\, ds+\beta^2 (1+\varepsilon) \int_0^T e^{-\gamma s}\| \nabla u_s -\nabla v_s \|^2\, ds,
\end{eqnarray*}

where $C$ , $\alpha$ and $\beta$ are the constants which appear in
the hypotheses of section \ref{hypo}. \\
Using the ellipticity assumption and taking the expectation, we obtain:
\begin{eqnarray*}
\begin{split}
(\gamma -1/\varepsilon ) &E( \int_0^T e^{-\gamma
s}\parallel\ub_s\parallel^2 \, ds)+(2\lambda -\alpha) E(\int_0^T e^{-\gamma
s}\| \nabla \ub_s \|^2\, ds)\leq \\&\hspace{-1cm}C(1+\varepsilon+2/\varepsilon) E(\int_0^T
e^{-\gamma s} \parallel u_s -v_s \parallel^2 \, ds)
+(C\varepsilon+\alpha+\beta^2 (1+\varepsilon) ) E(\int_0^T e^{-\gamma s}\|\nabla u_s -\nabla v_s \|^2\, ds).
\end{split}
\end{eqnarray*}

Now, we choose  $\varepsilon$ small enough and then $\gamma$ such
that
\[ C\varepsilon +\alpha + \beta^2 (1+\varepsilon) <2\lambda -\alpha \makebox{ and } \frac{\gamma-1/\varepsilon}
{2\lambda -\alpha}=\frac{C(1+\varepsilon+2/\varepsilon) }{C\varepsilon+\alpha+\beta^2 (1+\varepsilon ) }.\] If we set
$\delta= \frac{\gamma-1/\varepsilon} {2\lambda -\alpha}$, we have the following inequality:
\[ \forall u,v\in F_T^2 ,\, \parallel\Lambda (u)-\Lambda
(v)\parallel_{\gamma,\delta}\leq \frac{C\varepsilon +\alpha + \beta^2 (1+\varepsilon)
}{2\lambda -\alpha}\parallel u-v\parallel_{\gamma ,\delta}.\] We conclude thanks
to the fixed point Theorem and estimate \eqref{priori:estimate}.
\subsection{The truncation sequence $(\phi_n )_n$}
We first denote $\rho (x)$ the distance from a point $x\in\cO$ to the boundary of $\cO$, $\partial \cO$.\\
Following standard construction, for any $n\in\Ne$, we can construct a function $\phi_n \in C_0^{\infty} (\cO )$ satisfying the following properties:
\begin{enumerate}
\item $0\leq \phi_n \leq 1$;
\item $\phi_n =1$ on $\{ x\in\cO,\ \rho (x)\geq \frac{1}{n}\}$;
\item $\phi_n =0$ on $\{ x\in\cO,\ \rho (x)\leq \frac{1}{2n}\}$;
\item $|\partial_x \phi_n |\leq 3n$.
\end{enumerate}
\begin{Le}{\label{Troncation}} Assume that $\partial \cO$ is Lipschitz. Let $w\in H^1_0 (\cO )$, then
$(\phi_n w)_n$ tends to $w$ in $H^1_0 (\cO )$. \\
Moreover there exists a constant $C>0$ such that
$$\forall w\in H^1_0 (\cO),\ \sup_n \| \phi_n w\|_{H^1_0 (\cO )}\leq C\| w\|_{H^1_0 (\cO )}.$$
\end{Le}
\begin{proof}
Let us prove the first assertion. Let $w\in H^1_0 (\cO )$.\\
It is clear that we juste have to prove that $(w\partial_x \phi_n )_n$ tends to $0$ in $L^2 (\cO )$.\\
But, we know that $\displaystyle\frac{w}{\rho}$ belongs to $L^2 (\cO )$ (see Theorem 1.4.4.4 p.29 in \cite{Grisvard}). So, we have
\begin{eqnarray*}
\lim_{n\rightarrow +\infty}\int_{\cO} |w(x)\partial\phi_n (x)|^2\, dx&=&\lim_{n\rightarrow +\infty} \int_{\{ \rho(x)\leq\frac{1}{n}\}} |w(x)\partial_x\phi_n (x)|^2\, dx\\
&\leq &\lim_{n\rightarrow +\infty} 9n^2\int_{\{ \rho(x)\leq\frac{1}{n}\}} |w(x)|^2\, dx\\
&\leq&\lim_{n\rightarrow +\infty} 9\int_{\{ \rho (x)\leq\frac{1}{n}\}} \left|\displaystyle\frac{w(x)}{\rho (x)}\right|^2\, dx\\
&=&0,
\end{eqnarray*}
which proves the first part of the Lemma.\\
The second assertion is a consequence of the Banach-Steinhaus Theorem.\\
\end{proof}
\section*{Acknowledgment} 
The authors wish to thank the anonymous referee for all the pertinent remarks he made and for having pointed out some missing references.

\end{document}